\documentclass[final]{siamltex}

\usepackage{latexsym, enumerate}
\usepackage{eepic}
\usepackage{epic}
\usepackage{graphicx}
\usepackage{color}
\usepackage{ifpdf}
\usepackage{amssymb,amsmath,epsfig, algorithm,algorithmicx,multirow}
\usepackage{amsfonts, dsfont}

\newcommand{\pK}{\partial K}
\newcommand{\tx}{\tilde{\xi}}
\newtheorem{remark}{Remark}[section]
\newtheorem{assumption}{\sc Assumption}[section]

\title{A resourceful splitting technique with applications to deterministic and stochastic multiscale finite element methods
\thanks{The work is funded by the Department of Energy at Los Alamos National Laboratory
under contracts  DE-AC52-06NA25396 and the DOE Office of Science
Advanced Computing Research (ASCR) program in Applied Mathematical Sciences.
}}

\author{ L. Jiang\thanks{Applied Mathematics and Plasma Physics,
Los Alamos National Laboratory, NM 87545 ({\tt ljiang@lanl.gov}).}
        \and M. Presho\thanks{Department of Mathematics,
Colorado State University, Fort Collins, CO 80523
 ({\tt presho@math.colostate.edu}). }
}

\begin{document}

\maketitle

\begin{abstract}
In this paper we use a  splitting technique to develop  new multiscale basis functions for the multiscale finite element method (MsFEM).  The multiscale basis functions
are  iteratively generated using a Green's kernel. The Green's kernel is based on the first  differential operator of the splitting.  The proposed MsFEM is applied to deterministic elliptic equations and stochastic elliptic equations, and we show that the proposed MsFEM can considerably reduce the dimension of the random parameter space for stochastic problems. By combining the method with sparse grid collocation methods, the need for a prohibitive number of deterministic solves is alleviated. We rigorously analyze the convergence of the proposed method for both deterministic and stochastic elliptic equations. Computational complexity discussions are also offered to supplement the convergence analysis.  A number of numerical results are presented to confirm the theoretical findings.
\end{abstract}



\pagestyle{myheadings}

\thispagestyle{plain}
\markboth{L. Jiang and M. Presho}{A splitting technique with applications to  MsFEMs}

\section{Introduction}
Many fundamental and practical scientific problems involve a wide range of length scales. Typical examples may
include subsurface flows and geophysical domains with microscopic structures. Because there exist both natural randomness and lack of knowledge about the physical properties, it is often necessary to incorporate uncertainties into the model inputs. One way to address the uncertainties is to model the random inputs as a random field/process, and in turn, such problems are often modeled as stochastic partial differential equations (SPDEs). Then the model's output can be accurately predicted by efficiently solving the associated  SPDEs. It is challenging to solve the SPDEs when the random inputs vary over multiple scales in space and contain inherent uncertainties. The interest in developing stochastic multiscale methods for the SPDEs has steadily grown in recent years (see e.g., \cite{deh08, gz09, gmp10, jml10, wheeler}).

Let  $\Omega$  be a set of outcomes and $D$ be a bounded domain in $\mathbb{R}^d$ with a Lipschitz boundary.
We consider the stochastic elliptic boundary value problem:
seek a random field $u(x, \omega): \bar{D}\times \Omega\longrightarrow \mathbb{R}$  such that
$u(x, \omega)$  almost surely (a.s) satisfies the following equation
\begin{eqnarray}
\label{model-spde}
\begin{cases}
\begin{split}
-\nabla \cdot (k(x,\omega)\nabla u(x,\omega))&= f(x)  \ \ \text{in} \ \  D\\
u(x,\omega)&=0  \ \ \text{on} \ \  \partial D,
\end{split}
\end{cases}
\end{eqnarray}
where $k(x,\omega)$ is a scalar random field.
In particular, we assume that $k(x, \omega)$ exhibits  heterogeneity in multiple
scales over space. Since $k(x,\omega)$ varies over different spatial scales, resolving the finest scale  is not computationally feasible. Thus, we use
 multiscale methods. In practice, a high dimensional random field can be used to approximate the stochastic field $k(x, \omega)$, yet computing the statistical  output quantities of interest remains a difficult task.

During the last decade several multiscale methods have been developed, see e.g.~\cite{akl06, ab05, mortar, bl11,ee03,Hou, HughesII,jennylt03}.
The idea of multiscale methods is to divide the fine scale field into many local sub-problems and solve these in order to form a global coarse scale equation. This
leads to a coarse scale equation in which the fine scale effects are taken into account. One such multiscale method is the Multiscale Finite Element Method (MsFEM) \cite{Hou}. The main idea of MsFEM is to incorporate the small-scale information into the finite element basis functions and capture their effects  on the large scale through the discrete variational formulation. In many cases, the multiscale basis functions can be pre-computed  and used repeatedly in subsequent computations with different source terms, boundary conditions and even  modified coefficients.

The goal in this paper is to quantify the uncertainty through computing the statistical moments (e.g., expectation and variance) of the stochastic solution.
We note that the variance of the solution gives a measure for confidence of the solution expectation. Numerical solution strategies for stochastic PDEs generally  follow three main steps. First, the random inputs are approximately parameterized by a finite number of random variables. This can be achieved by a truncated Karhunen-Lo\`{e}ve expansion and/or truncated polynomial chaos expansion \cite{xiu10}. Second, a numerical approximation for the resulting high-dimensional deterministic PDE is used to approximate the solution with the respective input parameters. Finally, the solution is reconstructed as a random field and the statistical quantities of interest are computed. The second step is most difficult because the PDEs involve high-dimensional random parameter inputs. There exist many  methods for the second step. A broad survey of these methods can be found in \cite{Kleiber, xiu10}. Among these methods, Monte Carlo and stochastic collocation methods have been extensively studied and widely used. Monte Carlo methods and  stochastic collocation methods generate completely decoupled systems, each of which is the same size as the deterministic system.  This is suitable
for parallel computing and amendable  for relatively high-dimensional random inputs. In a Monte Carlo context, a large number of samples are randomly chosen and separate solves for each of the samples are used to determine the statistical behavior of solutions. However, a limitation is that convergence of Monte Carlo methods is usually slow.  Unlike Monte Carlo methods, stochastic collocation requires independent solves at fixed collocation points which are specifically chosen. In turn, this type of method has the capability to provide better accuracy than Monte Carlo with a fewer number of realizations. Moreover,  to overcome the {\it curse of dimensionality} imposed by high-dimensional input parameters, we can use Smolyak sparse  grids (see e.g., \cite{bnr00,ntw08,sm63}) to reduce the number of collocation points. In this paper, we consider the Monte Carlo method and the Smolyak sparse grid collocation method for stochastic approximation.

In the paper, we consider both multiscale features and uncertainties simultaneously. A main focus is the use of a resourceful splitting technique to compute MsFEM basis functions. For the problem (\ref{model-spde}), we assume that the coefficient $k$ can be split into two parts, $k=k_0 + k_1$. We then construct Green's functions using the differential operator associated with $k_0$.  The Green's functions are used to construct a sequence of multiscale ``bubble functions,'' which are employed to build the multiscale basis functions for MsFEM in an iterative manner. The Green's function technique provides an modified framework to compute the bubble functions, and is suitable for parallel computing due to the independent construction. The splitting of $k$ is flexible and can be easily controlled to lead to fast convergence of the bubble function sequence. Compared to standard MsFEM \cite{Hou}, the proposed MsFEM approach can accurately approximate multiscale solutions. The new multiscale approach is applied to SPDEs and may result in  new stochastic multiscale methods. Since the Green's functions essentially generate the MsFEM basis functions, this will reduce the dimension of random parameter space if the dimension of the random field for $k_0$ is smaller than that of the random field for $k$.  We note that using Karhunen-Lo\`{e}ve expansion or polynomial chaos expansion
usually yields an inherent splitting. The new MsFEM can efficiently solve SPDEs with high-dimensional input parameters, and combining the approach with sparse grid collocation methods alleviates the need for a prohibitive number of deterministic solves. We present convergence analysis of the proposed  MsFEM approach for deterministic elliptic equations and stochastic elliptic equations. Complexity analysis is also presented for deterministic MsFEM basis functions and stochastic MsFEM basis functions.

The rest of the paper is organized as follows. In Section $2$ we present the splitting technique which is used to compute the new MsFEM basis functions for  deterministic elliptic PDEs, and provide the associated computational algorithm. In Section $3$,  convergence analysis is rigorously derived for deterministic elliptic PDEs. Section $4$ is devoted to the applications  to  stochastic elliptic PDEs. We present convergence analysis and complexity analysis using  stochastic collocation methods in the section. In Section $5$, a number of numerical examples are presented to confirm the theoretical results. Some conclusions and closing remarks are made in Section $6$.

\section{A new approach for  MsFEM basis functions}
We consider the deterministic elliptic equation
\begin{eqnarray}
\label{model-eq}
\begin{cases}
\begin{split}
-\nabla \cdot (k\nabla u)&=f \quad \text{in}  \quad D\\
u&=0 \quad \text {on}   \quad \partial D,
\end{split}
\end{cases}
\end{eqnarray}
where  $k$ is a heterogenous scalar function. We note that our method can immediately be extended to the case of tensor coefficient function.
We assume that $k$ admits the splitting,
\begin{equation}
\label{k-split}
k = k_0 + k_1,
\end{equation}
where  $k(x)$ and $k_0(x)$ are bounded below and above, specifically,
\[
0<a_0 \leq k(x) \leq a_1, \quad 0<b_0 \leq k_0(x)\leq b_1  \quad \forall x\in D.
\]
Here, $k_0$ often represents the coarse scale information of $k$, and $k_1$ the fine scale information of $k$.

The multiscale finite element method (MsFEM) for Eq.~(\ref{model-eq}) was introduced in \cite{Hou} and  further analyzed in \cite{HouII}.
The key ingredient of MsFEM is the construction of an appropriate multiscale finite dimensional space in which the solution is sought. In particular, the fine scale heterogeneity in $k$ will be imbedded in this finite dimensional space. This information is incorporated into the coarse scale formulation through the coarse scale stiffness matrix. In this section, we develop a MsFEM basis function, which is constructed in a different way from the previous works (e.g., \cite{bco94, Hou}).

We introduce some notation for presentation. $L^p(D)$  ($1\leq p\leq \infty$) denotes the Lebesgue space. The norm of $L^2(D)$ is denoted by $\|\cdot\|_{0, D}$.
 $H^{1}(D)$ is the usual Sobolev space equipped with norm $\|\cdot\|_{1, D}$ and seminorm $|\cdot|_{1,D}$.
In the paper, $(\cdot, \cdot)$ is the usual $L^2$ inner product. We define an energy norm on a sub-domain $D'$ by $|||v|||_{D'}^2 :=(k\nabla v, \nabla v)_{D'}=\|\sqrt{k}\nabla v\|_{0,D'}^2$.
If $D'=D$, then $|||\cdot|||$ simply represents $|||\cdot|||_{D}$. We let $\mathfrak{T}_h$ be a quasi-uniform partition of $\Omega$ and $K$ be a representative coarse mesh with $\text{diam}(K)=h_K$. Let $h=\max\{h_K, K\in \mathfrak{T}_h\}$.


\subsection{Series approximation of multiscale basis functions}

Following \cite{Hou}, we define the standard multiscale basis functions by $\phi_{K, i}$ for vertices
$i=1,\dots,d$ of coarse cell $K$, which satisfy
\begin{eqnarray}
\label{MsFE-basis1}
\begin{cases}
\begin{split}
-\nabla \cdot  (k\nabla \phi_{K,i}) &= 0  \quad \text{in}  \quad K\\
\phi_{K,i}  &= l_{K,i}|_{\partial K}  \quad \text {on}   \quad \partial K,
\end{split}
\end{cases}
\end{eqnarray}
where $l_{K,i}$ is the boundary condition  associated with node $i$.
There exist some options for the boundary condition $l_{K,i}$ (see \cite{HouII,victor, jeg07}). Eq. $(\ref{MsFE-basis1})$
defines basis functions for local MsFEM if $l_{K,i}$ is a linear/bilinear  function. Incorporating global
information into $l_{K, i}$ produces global MsFEM \cite{jeg07}. We define the finite element  space for the standard MsFEM by
\[
V_h=\text{span} \{ \phi_{K,i}: i=1,...,d; K\in \mathfrak{T}_h \}.
\]
We note  that the idea of using basis functions satisfying certain
differential equations has been used before, see e.g.~ \cite{bco94,jem10}
and the references therein. Since we discuss a generic multiscale basis function, hereafter, we will remove the subindex $K$ and $i$ from  (\ref{MsFE-basis1}) for simplicity of presentation.

Next we use the splitting (\ref{k-split}) to derive a new MsFEM basis function. On each coarse cell $K\in \mathfrak{T}_h$, we define a projection operator $\Pi: H^1(K)\longrightarrow H_0^1(K)$ by
\begin{equation}
\label{def-pi}
(k_0\nabla \Pi v, \nabla w)=(k_0\nabla v, \nabla w)  \quad  \text{$\forall v\in H^1(K)$  and  $\forall w\in H_0^1(K)$}.
\end{equation}
The definition of $\Pi$ implies $\|\sqrt{k_0}  \nabla \Pi v\|_{0, K} \leq \|\sqrt{k_0} \nabla v\|_{0,K}$.

We extend the function for the boundary condition in (\ref{MsFE-basis1}) onto $K$ and denote the extended function by $l$.
 Let $\phi=(I-\Pi)l+\xi$, where $I$ is the identity operator. Then by (\ref{MsFE-basis1})  we can derive an equation for $\xi$
\begin{eqnarray}
\label{eq-xi}
\begin{cases}
\begin{split}
-\nabla \cdot (k\nabla \xi) &=-\nabla \cdot  (k_1 \nabla (\Pi-I) l) \quad \text{in}  \quad K\\
\xi &=0  \quad \text {on}   \quad \pK.
\end{split}
\end{cases}
\end{eqnarray}
We are going to construct a series to approximate $\xi$.
To this end, we set $\tilde{\xi}_0$ to satisfy
\begin{eqnarray}
\label{eq-xi0}
\begin{cases}
\begin{split}
-\nabla \cdot (k_0 \nabla \tilde{\xi}_0 ) &=-\nabla \cdot  (k_1 \nabla (\Pi-I) l) \quad \text{in}  \quad K\\
\tilde{\xi}_0 &=0  \quad \text {on}   \quad \pK.
\end{split}
\end{cases}
\end{eqnarray}
Then we recursively define a sequence of function $\tx_j$, $j=1,2,3,\cdots $  which satisfies
\begin{eqnarray}
\label{eq-xi-j}
\begin{cases}
\begin{split}
-\nabla \cdot (k_0\nabla \tx_j ) &=\nabla \cdot  (k_1 \nabla \tx_{j-1}) \quad \text{in}  \quad K\\
\tx_j  &=0  \quad \text {on}   \quad \pK.
\end{split}
\end{cases}
\end{eqnarray}
The function $\Pi l$ and the sequence $\{\tx_j\}$ are ``bubble functions" containing microstructure information, which are localized to a coarse cell by imposing zero
Dirichlet boundary conditions.
Let $\xi_J=\sum_{j=0}^J \tx_j$. We define the new multiscale basis function $\phi_J:=(I-\Pi)l+\xi_J$ and
the finite element  space for the new MsFEM  by
\[
V_{J, h}=\text{span} \{ (\phi_J)_{K,i}: i=1,...,d; K\in \mathfrak{T}_h \}.
\]


\subsection{Computational approach for the new MsFEM basis functions}
\label{eff-comp}

Since the proposed  multiscale basis function is defined as $\phi_J=(I-\Pi)l+\sum_{j=0}^J \tx_j$,
the computation of $\phi_J$  depends on the construction of $\Pi l$ and the bubble sequence $\{\tx_j\}_{j=0}^J$.
We find that $\Pi l$ and $\tx_j$ ($j=0, \cdots, J$) are all associated with the differential operator $\mathcal{L}_0:=-\nabla \cdot k_0 \nabla$.
Moreover, $\Pi l$ or $\tx_j$ ($j=0, \cdots, J$) can be formally written as $\mathcal{L}_0^{-1} \tilde{f}$, where $\tilde{f}$ is
the source term in the equation on $\Pi l$ or $\tx_j$.  It is well-known that Green's function can be viewed as generalized inverses
of differential operators.  We use Green's functions to obtain $\mathcal{L}_0^{-1}$ in the present work.

As far as making an efficient implementation, we use the Green's
function $G(x, y)$ for the operator $\mathcal{L}_0$. The Green's function $G(x,y)$ solves the equation
\begin{eqnarray}
\label{eq-green}
\begin{cases}
\begin{split}
-\nabla \cdot  ( k_0 \nabla G(x,y) ) &= \delta(x,y) \quad \text{in}  \quad K \\
G(x,y) &= 0 \quad \text {on}   \quad \pK.
\end{split}
\end{cases}
\end{eqnarray}
Since the Green's function $G(x, y)$ offers the fundamental solution for the differential operator
$\mathcal{L}_0$,  the Green's function $G(x, y)$ can efficiently generate $\Pi l$ and $\tx_j$ ($j=0, \cdots, J$).

By Eq.~(\ref{def-pi}), we have
\begin{equation}
\label{pi-l}
\Pi l (x)=-\int_K \, G(x,y)  \nabla_y \cdot  ( k_0 \nabla_y l) dy=\int_K k_0 \nabla_y G(x,y)\cdot \nabla_y l dy.
\end{equation}
Then we similarly  compute $\tilde{\xi}_0$ by performing
\[
\tilde{\xi}_0(x) = -\int_K \, G(x,y)  \nabla_y \cdot  \big( k_1  \nabla_y(\Pi-I) l\big) dy
= \int_K \, k_1 \nabla_y G(x,y) \cdot \big(\nabla_y (\Pi-I)l\big) dy,
\]
and compute $\tilde{\xi}_k$, $k=1,2,3,\cdots$,  by performing
\[
\tx_j(x) = \int_K \, G(x,y) \nabla_y \cdot \big( k_1  \nabla_y \tx_{j-1}\big)  dy
=- \int_K \, k_1 \nabla_y G(x,y) \cdot \nabla_y \tx_{j-1} dy.
\]

To discuss the complexity of computation for the proposed MsFEM basis function, we investigate the computation in terms of matrix operations.
We use the vector function $\vec{b}(x)=(\ell_1(x), \cdots, \ell_{n_K}(x))^T$, where $\ell_p(x)$ ($p=1, \cdots n_K$) is a standard finite element basis function at the underlying  vertex $x_p$ of a underlying fine grid in  $K$. Here $n_K$ is the number of the internal fine vertices in $K$.  We define vectors $v_0$ and $v_1$ by
\begin{equation}
v_0 =\int_K k_0 \nabla \vec{b}\otimes \nabla l dx  \quad \text{and} \quad v_1 = \int_K k_1 \nabla \vec{b} \otimes \nabla l dx  \label{def-v01},
\end{equation}
where $\otimes$ represents the tensor product.
We use $\mathcal{L}_1=-\nabla\cdot k_1\nabla$ to denote the differential operator associated with $k_1$. Let $M_0$ and $M_1$  be the stiffness matrices associated with the operators $\mathcal{L}_0$ and $\mathcal{L}_1$, respectively. Then

\begin{equation}
M_0=\int_K k_0 \nabla \vec{b}\otimes \nabla \vec{b} dx \quad \text{and} \quad M_1=\int_K k_1 \nabla \vec{b}\otimes \nabla \vec{b} dx  \label{def-M01}.
\end{equation}
We have the following theorem to represent $\Pi l$ and $\tx_j$ ($j=0, \cdots, J$) in the finite element space of fine grid.
The notations $\Pi l$ and $\tx_j$ are slightly abused in the following theorem.
\begin{theorem}
\label{thm-v-M}
Let $\Pi l(x)$ and $\tx_j(x)$ ($j=0, \cdots, J$) be the finite element approximations  on the underlying fine grid in $K$. Then
\begin{equation}
\label{pi-l-fe}
\Pi l(x)=(M_0^{-1} \vec{b}(x))^T v_0
\end{equation}
and for $j=0,\cdots, J$,
\begin{equation}
\label{tx-fe}
\tx_j (x)=(-1)^j \big(M_0^{-1} \vec{b}(x)\big)^T (M_1 M_0^{-1})^j (M_1M_0^{-1} v_0-v_1).
\end{equation}
\end{theorem}
\begin{proof}
Let us still use  $G(x,y)$ to represent the numerical Green's function on the underlying fine grid.
Then direct calculation implies that
\begin{equation}
\label{comp-G}
G(x, y)=\big(M_0^{-1} \vec{b}(x)\big)^T \vec{b}(y).
\end{equation}
Thanks to Eq.~(\ref{pi-l}) and Eq. ~(\ref{comp-G}), it follows that
\begin{eqnarray}
\label{comp-pi-l}
\begin{split}
\Pi l(x)&=\int_K k_0\nabla_y G(x,y)\cdot \nabla_yl dy=\int_K k_0 \nabla_y \big( M_0^{-1}\vec{b}(x)\big)^T\vec{b}(y)\cdot \nabla_y l dy\\
&= \big( M_0^{-1}\vec{b}(x)\big)^T \int_K k_0\nabla_y \vec{b}\otimes \nabla_y l dy=\big( M_0^{-1}\vec{b}(x)\big)^T v_0.
\end{split}
\end{eqnarray}
For $\tx_0(x)$,  we have
\begin{eqnarray}
\label{comp-tx0}
\begin{split}
\tx_0(x) &=\int_K k_1\nabla_y G(x,y)\cdot \nabla_y \Pi l(y) dy -\int_K k_1\nabla_y G(x,y)\cdot \nabla_y ldy\\
&=\int_K k_1 \nabla_y \big(M_0^{-1}\vec{b}(x)\big)^T \vec{b}(y)\cdot \nabla_y \big(M_0^{-1} \vec{b}(y)\big)^T v_0 dy-\big(M_0^{-1} \vec{b}(x)\big)^T
\int_K k_1\nabla_y \vec{b}(y)\otimes \nabla_y l dy\\
&=\big(M_0^{-1} \vec{b}(x)\big)^T \big [\int_K k_1 \nabla_y \vec{b}(y)\otimes \nabla_y \vec{b}^T(y) dy \big ]M_0^{-1} v_0-\big(M_0^{-1} \vec{b}(x)\big)^T v_1\\
&=\big(M_0^{-1} \vec{b}(x)\big)^T (M_1M_0^{-1} v_0- v_1).
\end{split}
\end{eqnarray}
Using Eq. ~(\ref{comp-tx0}), we obtain
\begin{eqnarray}
\label{comp-tx1}
\begin{split}
\tx_1(x)&=-\int_K k_1 \nabla_y G(x,y)\cdot \nabla_y\tx_0(y)dy\\
&=- \big(M_0^{-1} \vec{b}(x)\big)^T\big[\int_K k_1 \nabla_y \vec{b}(y)\otimes \nabla_y \vec{b}^T (y)dy\big] M_0^{-1}(M_1 M_0^{-1} v_0-v_1)\\
&=- \big(M_0^{-1} \vec{b}(x)\big)^T(M_1 M_0^{-1})(M_1 M_0^{-1} v_0-v_1).
\end{split}
\end{eqnarray}
By repeating the procedure of (\ref{comp-tx1}),  it follows immediately that for $j=2, \cdots, J$,
\[
\tx_j(x)=(-1)^j \big(M_0^{-1} \vec{b}(x)\big)^T(M_1 M_0^{-1})^j(M_1 M_0^{-1} v_0-v_1).
\]
The proof is complete.
\end{proof}

We pre-compute  vectors $v_0$, $v_1$ and matrices $M_0$, $M_1$.  Since $v_0$, $v_1$, $M_0$ and $M_1$ only depend on
the local information in $K$, their construction is suitable for parallel computation.
By Theorem \ref{thm-v-M},  we obtain  $\Pi l(x)$ and $\tx_j(x)$ ($j=0,\cdots, J$)  by performing a direct matrix-vector multiplication.
Moreover, we find that the computations of $\Pi l(x)$ and $\tx_j(x)$ ($j=0,\cdots, J$) are independent of each other  and
suitable for parallel computation as well. By Theorem \ref{thm-v-M}, the numerical representation of $\phi_J$ can be written as
\begin{equation}
\label{comp-new-MsFE}
\phi_J(x)=l(x)- \big(M_0^{-1} \vec{b}(x)\big)^T v_0 +\sum_{j=0}^J (-1)^j \big(M_0^{-1} \vec{b}(x)\big)^T(M_1 M_0^{-1})^j(M_1 M_0^{-1} v_0-v_1).
\end{equation}
If we still use $\phi(x)$ to denote the numerical approximation  of a standard MsFEM basis function in the underlying
fine grid of $K$, then it is easy to show that
\begin{equation}
\label{comp-MsFE}
\phi(x)=l(x)- \big( M^{-1} \vec{b}(x)\big )^T v,
\end{equation}
where $M=\int_K k \nabla \vec{b}\otimes \nabla \vec{b} dx$  and $v=\int_K k \nabla \vec{b}\otimes \nabla l dx$.
Compared Eq. ~(\ref{comp-new-MsFE}) and Eq. ~(\ref{comp-MsFE}), we find that
the computation of $\phi_J (x)$ is comparable to the computation of $\phi(x)$ in a parallel setting
since each of the terms in $\phi_J(x)$ may be obtained independently. When $k_0=k_1$,  Eq. ~(\ref{comp-new-MsFE}) and Eq. ~(\ref{comp-MsFE}) imply that
$\phi_J(x)=\phi(x)$,  which means that the proposed MsFEM coincides with standard MsFEM.


Application of this procedure is summarized in Algorithm \ref{alg:det-msfem}.
\begin{algorithm}[h]\caption{Computing  the proposed  MsFEM basis function $\phi_J$}\label{alg:det-msfem}
\begin{enumerate}
\item Assemble the vectors $v_0$ and $v_1$ by (\ref{def-v01}), and assemble the matrices $M_0$ and $M_1$ by (\ref{def-M01});
\item Compute each term of (\ref{comp-new-MsFE}) independently;
\item Construct the proposed MsFEM  basis function  $\phi_J(x)$ by summing the terms obtained in step (2).
\end{enumerate}
\end{algorithm}

\begin{remark}
Let $\xi$ be the finite element solution (i.e., evaluated at fine vertices) of Eq.~(\ref{eq-xi}). Here the notation $\xi$ is slightly abused. Then by Eq.~(\ref{eq-xi}) we have
$(M_0+M_1)\xi=M_1 M_0^{-1}v_0-v_1$. This gives the Neumann series of $\xi$,
\begin{equation}
\label{xi-neum}
\xi=\sum_{j=0}^{\infty} (-1)^jM_0^{-1}(M_1 M_0^{-1})^j (M_1 M_0^{-1} v_0-v_1).
\end{equation}
Consequently,  the last term in Eq.~(\ref{comp-new-MsFE}) is actually the truncated Neumann series.
By the Neumann series and Eq.~(\ref{comp-new-MsFE}), if the spectral radius of $(M_1 M_0^{-1})$ is less than $1$, then $\phi_J$ is convergent as $J \rightarrow \infty$.
This will be consistent with {Assumption 3.1}.
\end{remark}

\section{Convergence analysis for deterministic coefficient }
\label{converge}
We make the following assumption for convergence analysis.
\begin{assumption}
{\it Assume that the splitting $k=k_0+k_1$ on $K$ satisfies
\begin{equation}
\label{assum1}
\eta_K:=\|{k_1\over k_0}\|_{L^{\infty}(K)}< 1.
\end{equation}
}
\end{assumption}
If the property of the splitting $\|{k_1\over k_0}\|_{L^{\infty}(K)}<1 $ fails,
then we can introduce a slightly modified splitting of $k$ such that the assumption (\ref{assum1}) is satisfied.
As a matter of fact, we have the following proposition.
\begin{proposition} \label{sprop}
Suppose that $\|{k_1\over k_0}\|_{L^{\infty}(K)}\geq 1$. Let constant $s$ satisfy the following inequality
\begin{equation}
\label{s-eq}
s>\sup_{x\in K} \max \left( \frac{k_1(x)-k_0(x)}{2}, -k_0(x) \right).
\end{equation}
Then the modified splitting $k=\tilde{k}_0+\tilde{k}_1:=(k_0+s)+(k_1-s)$
satisfies the assumption (\ref{assum1}) if we switch $(k_0, k_1)$ to $(\tilde{k}_0, \tilde{k}_1)$ in (\ref{assum1}).
\end{proposition}
\begin{proof}
If $\|{k_1\over k_0}\|_{L^{\infty}(K)}\geq 1$ and (\ref{s-eq})  holds, then direct calculation implies that
$\|\frac{\tilde{k}_1}{\tilde{k}_0}\|_{L^{\infty}(K)}< 1$. Consequently,
the assumption (\ref{assum1}) is satisfied.
\end{proof}

We note that the modified splitting, $k=\tilde{k}_0+\tilde{k}_1$,  is essentially the same (up to a constant) as the original splitting, $k=k_0+k_1$.

Now we provide the main convergence result for deterministic MsFEMs.
\begin{theorem}
\label{det-thm}
Define $\tilde{c}=\sqrt{2}\max_{K\in \mathfrak{T}_h} \|\sqrt{k_0\over k}\|_{L^{\infty}(K)}$ and
 $\eta=\max_{K\in \mathfrak{T}_h} \eta_K< 1$.     Let $u_h\in V_h$ and $u_{J, h}\in V_{J,h}$ be the MsFEM solution for Eq. (\ref{model-eq}). Then \\
(1) $|||u_h-u_{J, h}|||\leq (\tilde{c})^{1/2} \big(\eta^{J+1}+\frac{\eta^{J+1}}{1-\eta^{J+1}}\big )^{1/2} |||u|||$.\\
(2) $|||u-u_{J, h}|||\leq  \sqrt{2}\tilde{c} \big(\eta^{J+1}+\frac{\eta^{J+1}}{1-\eta^{J+1}}\big ) |||u||| + 2|||u-u_h|||$.
\end{theorem}

Thanks to the boundedness of $k_0$ and $k$, it follows that $\tilde{c}\leq \sqrt{\frac{2b_1}{a_0}}$.
By Theorem \ref{det-thm}, it follows that
\[
\lim_{J\rightarrow \infty}|||u_h-u_{J, h}|||=0, \quad  \lim_{h\rightarrow 0} \lim_{J\rightarrow \infty}|||u-u_{J, h}|||=0.
\]
This shows convergence of the proposed  MsFEM.

The  arguments of the proof Theorem \ref{det-thm} include the following  Lemma \ref{lem-pi-2},  Lemma \ref{lem-pi-3}  and Lemma \ref{lem-pi-4}.
In the rest of the section, we will formulate these lemmas.

To describe  and prove Lemma \ref{lem-pi-2}, we need a couple of lemmas.
The following lemma gives an upper bound of the sequence $\{\tx_j\}$.
\begin{lemma}
\label{lem-xi-j}
Let $\tx_0$ and $\tx_j$ ($j=1, 2,\cdots $) solve Eqs.~(\ref{eq-xi0}) and (\ref{eq-xi-j}), respectively.  Then
\[
\|\sqrt{k_0} \nabla \tx_j\|_{0,K} \leq 2 \eta_K^{j+1} \|\sqrt{k_0} \nabla l\|_{0,K},    \quad j=0,1,2, \cdots,
\]
where $\eta_K$ is defined in (\ref{assum1}).
\end{lemma}

The proof of Lemma \ref{lem-xi-j} is presented in Appendix \ref{app1}.

The following lemma shows that the sequence of the new MsFE basis functions converge to a standard  MsFE basis function.
\begin{lemma}
\label{singlethm}
The sequence of basis functions  $\{\phi_J\}$ is convergent. Moreover,
\[
\lim_{J\rightarrow \infty}|||\phi_J-\phi|||_{K} =0,
\]
where $\phi$ solves the standard MsFE basis equation (\ref{MsFE-basis1}).
\end{lemma}

The proof of Lemma \ref{singlethm} is presented in Appendix \ref{app2}.

From the proof of Lemma \ref{singlethm}, we can obtain an explicit convergence rate for $\phi_J$, which is stated as following:
\begin{proposition}
\label{prop3.2}
If the coarse function  $l$ in Eq.~(\ref{eq-xi0}) is in $H^1(K)$, then
\[
|||\phi-\phi_J|||_K  \leq C_l\frac{2 b_1}{\sqrt{a_0}} \eta_K^{J+2},
\]
where the constant $C_l=\|\nabla l\|_{0,K}$.
\end{proposition}

The proof of Proposition \ref{prop3.2} is presented in Appendix \ref{app3}.

To describe Lemma \ref{lem-pi-2}, we need to introduce elliptic projection operators. Specifically, we define the elliptic projection $\Pi_h: H_0^1(D)\longrightarrow V_h $ by
\[
(k\nabla \Pi_h v, \nabla w):=(k\nabla v, \nabla w) \quad \text{$\forall v\in H_0^1(D)$  and $\forall w\in V_h$}.
\]
  Similarly,  the
elliptic projection $\Pi_{J,h}: H_0^1(D) \longrightarrow V_{J,h} $ is defined  by
\[
(k\nabla \Pi_{J,h} v, \nabla w):=(k\nabla v, \nabla w) \quad \text{$\forall v\in H_0^1(D)$  and $\forall w\in V_{J,h}$}.
\]
Since $\lim_{J\rightarrow \infty} V_{J,h}=V_h$ by Lemma \ref{singlethm},  the standard density argument \cite{bs94} implies that  $\lim_{J\rightarrow \infty} \Pi_{J,h}=\Pi_h$.
We can show that the projection operators $\Pi_h$ and $\Pi_{J, h}$ are self-adjoint and  idempotent operators.
Let $u_h$ and $u_{J,h}$ be the MsFEM solution in $V_h$  and $V_{J,h}$, respectively,  for Eq.  (\ref{model-eq}). Then
\[
u_h=\Pi_h u \quad \text{and}  \quad u_{J,h}=\Pi_{J,h} u.
\]

Using Lemma \ref{lem-xi-j}  and the technique presented in \cite{gmp10}, we have the following result.
\begin{lemma}
\label{lem-pi-2}
Let $u\in H_0^1(D)$. Then
\begin{equation}
\label{est-pi2}
|||(I-\Pi_{J, h})\Pi_h u|||+|||(I-\Pi_h)\Pi_{J, h}u|||\leq \tilde{c} \big(\eta^{J+1}+\frac{\eta^{J+1}}{1-\eta^{J+1}}\big ) |||u|||,
\end{equation}
where $\tilde{c}$ and $\eta$ are defined in Theorem \ref{det-thm}.
\end{lemma}

\begin{proof}
For any $u\in H_0^1(D)$, we write
\[
\Pi_h u=\sum_K \sum_i \alpha_{K,i}\phi_{K,i}, \quad \Pi_{J,h} u=\sum_K\sum_i \beta_{K,i}(\phi_J)_{K,i}.
\]
We define the operator $T: H_0^1(K)\longrightarrow H_0^1(K)$ by
\[
-\nabla \cdot (k_0 \nabla Tv)=\nabla \cdot (k_1\nabla v), \quad \forall v\in H_0^1(K).
\]
Then we have
\begin{eqnarray}
\label{base-dif}
\begin{split}
\phi_{K,i}-(\phi_J)_{K,i}&=\sum_{j=J+1}^{\infty} (\tx_j)_{K,i}=\sum_{j=J+1}^{\infty} T^j (\tx_0)_{K,i}\\
&=T^{J+1} \sum_{j=0}^{\infty} T^j (\tx_0)_{K,i}=T^{J+1} (\phi_{K,i}-(I-\Pi)l_{K,i}).
\end{split}
\end{eqnarray}
By (\ref{base-dif}), we have
\begin{equation}
\label{eq1-lem3.3}
(I-\Pi_{J,h})\Pi_h u |_K=\sum_i \alpha_{K,i} \big( \phi_{K,i}-(\phi_J)_{K,i}\big) =T^{J+1} \big(\Pi_h u|_K-\sum_i \alpha_{K,i}(I-\Pi)l_{K,i}\big).
\end{equation}
and
\begin{equation}
\label{eq2-lem3.3}
(I-\Pi_h) \Pi_{J,h}u |_K =\sum_i \beta_{K,i}\big( (\phi_J)_{K,i}-\phi_{K,i} \big)=T^{J+1} \big(\sum_i \beta_{K,i} (I-\Pi)l_{K,i}-\Pi_h \Pi_{J,h} u|_K\big).
\end{equation}
We first estimate $(I-\Pi_{J,h})\Pi_h u$.  Let $z_K=\Pi_h u|_K-\sum_i \alpha_{K,i}(I-\Pi)l_{K,i}$.
Thanks to (\ref{eq1-lem3.3}) and the proof of Lemma \ref{lem-xi-j}, it follows that
\begin{eqnarray*}
\begin{split}
|||(I-\Pi_{J,h})\Pi_h u|||_K^2 &= \big( k\nabla (\Pi_h u-\Pi_{J,h}\Pi_h u), \nabla (\Pi_h u-\Pi_{J,h}\Pi_h u)\big)_K\\
&= \big (k \nabla (T^{J+1} z_K), \nabla (\Pi_h u-\Pi_{J,h}\Pi_h u)\big)_K\\
& \leq \|\sqrt{k} \nabla (T^{J+1} z_K)\|_{0, K} |||(I-\Pi_{J,h})\Pi_h u|||_K\\
&\leq \|\sqrt{k\over k_0}\|_{L^{\infty}(K)}\|\sqrt{k_0} \nabla (T^{J+1} z_K)\|_{0, K}|||(I-\Pi_{J,h})\Pi_h u|||_K\\
&\leq \sqrt{2} \eta_K ^{J+1} \| \sqrt{k_0}\nabla  z_K\|_{0, K}|||(I-\Pi_{J,h})\Pi_h u|||_K.
\end{split}
\end{eqnarray*}
This implies that
\begin{equation}
\label{eq3-lem3.3}
|||(I-\Pi_{J,h})\Pi_h u|||_K \leq \sqrt{2} \eta_K ^{J+1} \| \sqrt{k_0}\nabla  z_K\|_{0, K}.
\end{equation}
Since $z_K\in H_0^1(K)$ and $\big(k_0 \nabla (I-\Pi) l_{K,i}, z_K)=0$, we have
\begin{eqnarray*}
\begin{split}
&\|\sqrt{k_0}\nabla  z_K\|_{0, K}^2=\big (k_0 \Pi_h u|_K-k_0\sum_i \alpha_{K,i}(I-\Pi)l_{K,i}, \nabla z_K\big) \\
&=\big (k_0 \Pi_h u|_K, \nabla z_K\big)\leq \|\sqrt{k_0\over k}\|_{L^{\infty}(K)} |||\Pi_h u|||_K \|\sqrt{k_0}\nabla  z_K\|_{0, K},
\end{split}
\end{eqnarray*}
from which we get
\begin{equation}
\label{zk-bound}
\|\sqrt{k_0}\nabla  z_K\|_{0, K}\leq \|\sqrt{k_0\over k}\|_{L^{\infty}(K)} |||\Pi_h u|||_K.
\end{equation}
By (\ref{eq3-lem3.3}) and (\ref{zk-bound}), it follows immediately that
\begin{eqnarray}
\label{eq4-lem3.3}
\begin{split}
|||(I-\Pi_{J,h})\Pi_h u||| &\leq \sqrt{2} \max_K \left( \|\sqrt{k_0\over k}\|_{L^{\infty}(K)} \eta_K^{J+1} \right) |||\Pi_h u|||\\
&\leq  \tilde{c} \eta^{J+1} |||u|||.
\end{split}
\end{eqnarray}

Next we estimate $(I-\Pi_h) \Pi_{J,h}u$.
By (\ref{eq2-lem3.3}), we have
\begin{eqnarray}
\label{eq10-lem3.3}
\begin{split}
&\|\sqrt{k_0}\nabla (I-\Pi_h) \Pi_{J,h}u\|_{0,K} =\|\sqrt{k_0}\nabla T^{J+1} \big(\sum_i \beta_{K,i} (I-\Pi)l_{K,i}-\Pi_h \Pi_{J,h} u|_K\big)\|_{0,K}\\
&\leq \eta_K^{J+1} \left[ \|\sqrt{k_0} \nabla \big(\Pi_{J,h} u- \sum_i \beta_{K,i} (I-\Pi)l_{K,i}\big)\|_{0,K} + \|\sqrt{k_0} \nabla (\Pi_{J,h} u- \Pi_h \Pi_{J,h} u )\|_{0, K} \right].
\end{split}
\end{eqnarray}
Let $\tilde{z}_K=\Pi_{J, h} u|_K -\sum_i\beta_{K,i} (I-\Pi)l_{K,i}$. Then  $\tilde{z}_K\in H_0^1(K)$ and
\begin{equation}
\label{eq11-lem3.3}
(k_0 \nabla \tilde{z}_K, \nabla \tilde{z}_K) = (k_0 \nabla \Pi_{J,h} u, \nabla \tilde{z}_K)\leq \|\sqrt{k_0} \nabla \Pi_{J,h} u \|_{0,K} \|\sqrt{k_0}\nabla \tilde{z}_K\|_{0, K}.
\end{equation}
Combining (\ref{eq10-lem3.3}) and (\ref{eq11-lem3.3}) implies that
\[
\|\sqrt{k_0}\nabla (I-\Pi_h) \Pi_{J,h}u\|_{0,K} \leq \frac{\eta_K^{J+1}}{1-\eta_K^{J+1}} \|\sqrt{k_0} \nabla \Pi_{J,h} u \|_{0,K}.
\]
Because
$(k_0 \nabla \Pi_{J,h} u, \nabla \Pi_{J,h} u)\leq  \| {k_0\over k}\|_{L^{\infty}(K)} |||\Pi_{J,h} u|||_K^2$, it follows
\begin{equation}
\label{eq12-lem3.3}
\|\sqrt{k_0}\nabla (I-\Pi_h) \Pi_{J,h}u\|_{0,K} \leq \frac{\eta_K^{J+1}}{1-\eta_K^{J+1}} \| {k_0\over k}\|_{L^{\infty}(K)} |||\Pi_{J,h} u|||_K.
\end{equation}
Since $\|{k_1\over k_0}\|_{L^{\infty}(K)}<1$, we get $\| {k_0\over k}\|_{L^{\infty}(K)}>{1\over 2}$.
A direct calculation gives
\begin{equation}
\label{eq13-lem3.3}
\|\sqrt{k_0}\nabla (I-\Pi_h) \Pi_{J,h}u\|_{0,K}^2 > {1\over 2} |||(I-\Pi_h) \Pi_{J,h}u|||_K^2.
\end{equation}
By (\ref{eq12-lem3.3}) and (\ref{eq13-lem3.3}), we have
\[
|||(I-\Pi_h) \Pi_{J,h}u|||_K\leq \sqrt{2} \| {k_0\over k}\|_{L^{\infty}(K)} \frac{\eta_K^{J+1}}{1-\eta_K^{J+1}}|||\Pi_{J,h} u|||_K.
\]
Since $|||\Pi_{J,h} u|||_K\leq |||u|||_K$, it follows immediately that
\begin{equation}
\label{eq14-lem3.3}
|||(I-\Pi_h) \Pi_{J,h}u|||\leq \tilde{c} \frac{\eta^{J+1}}{1-\eta^{J+1}}||| u|||_K.
\end{equation}
Combining (\ref{eq4-lem3.3}) and (\ref{eq14-lem3.3}) completes the proof.
\end{proof}

A straightforward application of Theorem 3.6 in \cite{gmp10} gives rise to the following lemma.
\begin{lemma}
\label{lem-pi-3}
Let the inequality (\ref{est-pi2})  in Lemma \ref{lem-pi-2} hold. Then 
\[
|||u_h-u_{J, h}|||\leq (\tilde{c})^{1/2} \big(\eta^{J+1}+\frac{\eta^{J+1}}{1-\eta^{J+1}}\big )^{1/2} |||u|||,
\]
where $\tilde{c}$ is defined in Lemma  \ref{lem-pi-2}. 
\end{lemma}

Lemma \ref{lem-pi-3} gives the first convergence  result in Theorem \ref{det-thm}.

Following the proof of Theorem 3.8 in \cite{gmp10} and Lemma \ref{lem-pi-3}, we have the following lemma.
\begin{lemma} Let $u$ be the solution to Eq. (\ref{model-eq}) and $u_h\in V_h$ the standard MsFEM solution of Eq. (\ref{model-eq}). Then
\label{lem-pi-4}
\[
|||u-\Pi_{J,h}u|||\leq  \sqrt{2}\tilde{c} \big(\eta^{J+1}+\frac{\eta^{J+1}}{1-\eta^{J+1}}\big ) |||u||| + 2|||u-u_h|||.
\]
\end{lemma}

Because $u_{J,h}=\Pi_{J,h}u$, the second convergence result in Theorem \ref{det-thm} follows  Lemma \ref{lem-pi-4} immediately.

\section{Analysis for stochastic coefficient}
\label{stochanalysis}
We consider the stochastic elliptic equation (\ref{model-spde}). To make $k(x,\omega)$  positive,  we consider a logarithmic   stochastic field, $k(x, \omega):=\exp(Y(x, \omega))$,
where  $Y(x, \omega)$ is a stochastic field with second moment.

We assume that $Y(x, \omega)$ admits the following
 truncated Karhunen-Lo\`{e}ve expansion (see \cite{xiu10} for details), i.e.,
\begin{equation}
\label{KLE}
Y(x, \omega)=E[Y]+\sum_{i=1}^n \sqrt{\lambda_i} b_i(x) \theta_i(\omega).
\end{equation}
Here $
(b_i, b_j)_{L^2(D)}=\delta_{ij}, \ \  \lambda_1 \geq \lambda_2\cdots, \geq \lambda_m\cdots,
\ \  \lim_{m\longrightarrow \infty} \lambda_m =0.
$
Let $\Theta:=(\theta_1,\cdots, \theta_m, \theta_{m+1},\cdots, \theta_n):=  (\Theta_0, \Theta_1)\in \mathbb{R}^n$,
 where $\Theta_0:=(\theta_1,\cdots, \theta_m) \in \mathbb{R}^m$ and $\Theta_1\in \mathbb{R}^{n-m}$.
Then stochastic field $k(x, \omega)$ can be parameterized to a finite-dimensional random field $k(x, \Theta)$.
We define the splitting of $k(x,\Theta)$ by
\begin{equation}
\label{split-random-k}
k(x, \Theta)=k_0(x,\Theta_0)+k_1(x, \Theta),
\end{equation}
where $k_0(x,\Theta_0):=\exp(E[Y]+\sum_{i=1}^m \sqrt{\lambda_i} b_i(x) \theta_i(\omega))$ ($m<n$) and $k_1(x, \Theta)=k(x,\Theta)-k_0(x,\Theta_0)$.

\subsection{Convergence analysis}
In the subsection, we will present convergence analysis when the random field $k(x, \Theta)$ admits the splitting described  in (\ref{split-random-k}).

We define $L^2(\Omega)$ to be the square integrable space with the probability measure $\rho(\Theta)d\omega$,
where $\rho(\Theta)$ is the joint probability density function of $\Theta$.
We can use Theorem \ref{det-thm} to derive an error estimate of the stochastic elliptic equation (\ref{model-spde}).
\begin{theorem} \label{stochbound}
Let $u_h$ and $u_{J, h}$ be the MsFEM solution of stochastic elliptic equation  in space $V_h\times L^2(\Omega)$ and $V_{J,h}\times  L^2(\Omega)$, respectively, for the stochastic equation (\ref{model-spde}).
If $\eta:=\|\frac{k_1(x, \Theta)}{k_0(x,\Theta_0)}\|_{L^{\infty}(D\times \Omega)}< 1$, then
\[
 ||\sqrt{k}\nabla(u_h-u_{J, h})||_{L^2(D)\otimes L^2(\Omega)} \leq (\tilde{c})^{1/2} \big(\eta^{J+1}+\frac{\eta^{J+1}}{1-\eta^{J+1}}\big )^{1/2} ||\sqrt{k}\nabla u||_{L^2(D)\otimes L^2(\Omega)},
 \]
where  $\tilde{c}=\sqrt{2} \|\sqrt{\frac{k_0(x, \Theta_0)}{k(x, \Theta)}}\|_{L^{\infty}(D\times \Omega)}$.
\end{theorem}

The eigenvalues $\{\lambda_i\}$ play an important role to control $|{k_1\over k_0}|$. To this end, we define the
energy ratio $E(m)$ by $E(m)=\frac{\sum_{i=1}^m \sqrt{\lambda_i}}{\sum_i^n \sqrt{\lambda_i}}$.
Then we can show that  $\|\frac{k_1(x,\Theta)}{k_0(x, \Theta_0)}\|_{L^{\infty}(D\times \Omega)}$ is proportional to $1-E(m)$
under certain conditions.
\begin{proposition}
\label{k1/k0-random-prop}
Let $\|\theta_i\|_{L^{\infty}(\Omega)}\leq C_{\theta}$ uniformly for all $m+1\leq i\leq n$.
If $\text{cov}[Y](x_1,x_2)$ are piecewise analytic in $D\times D$, then
there exists constant a $C_Y$ such that for $m$   large enough
\[
\|\frac{k_1(x,\Theta)}{k_0(x,\Theta_0)}\|_{L^{\infty}(D\times \Omega)}\leq C_Y\big(1-E(m)\big),
\]
where $C_Y= {7\over 4} C_{\theta} \max_{m+1\leq i\leq n}  \{|b_i|_{L^{\infty}(D)} \}$.
\end{proposition}
\begin{proof}
Since $k_1=k_0 (\frac{k}{k_0}-1)=k_0\big( \exp(\log k-\log k_0)-1\big)$, then it follows that
\[
|\frac{k_1}{k_0}|=\left| \exp \left( \sum_{i=m+1}^n \sqrt{\lambda_i} b_i(x)\theta_i(\omega) \right)-1 \right|
\]
Because  $\text{cov}[Y](x_1,x_2)$ is piecewise analytic in $D\times D$, then there exist positive constants $C_0$  independent of $m$ and $d$  (see \cite{Schwab}) such that for given $0<s<{1\over 2}$
\begin{eqnarray}
\label{ineq01-prop4.1}
\| \sum_{i=m+1}^n \sqrt{\lambda_i} b_i(x)\theta_i(\omega)\|_{L^{\infty}(D\times \Omega)} \leq C_1 \exp\big(-C_0({1\over 2}-s)m^{1\over d}\big),
\end{eqnarray}
where $C_1=C_1(C_{\theta},s,d, C_0)$. If $m$ is large enough, inequality (\ref{ineq01-prop4.1}) implies that
\[
\| \sum_{i=m+1}^n \sqrt{\lambda_i} b_i(x)\theta_i(\omega))\|_{L^{\infty}(D\times \Omega)}\leq 1.
\]
Due to the inequality $|e^x-1|\leq {7\over 4} |x|$  for $|x|\leq 1$,
then  it follows
\begin{eqnarray}
\begin{split}
\|\frac{k_1}{k_0}\|_{L^{\infty}(D\times \Omega)} & \leq {7\over 4} \|\sum_{i=m+1}^n \sqrt{\lambda_i} b_i(x)\theta_i(\omega))\|_{L^{\infty}(D\times \Omega)}\\
&\leq {7\over 4} C_{\theta} \max_{m+1\leq i\leq n}  \{|b_i|_{L^{\infty}(D)} \}(\sum_{i=m+1}^n \sqrt{\lambda_i})\\
& \leq {7\over 4} C_{\theta} \max_{m+1\leq i\leq n}  \{|b_i|_{L^{\infty}(D)} \}(1-E(m)).
\end{split}
\end{eqnarray}
The proof is completed.
\end{proof}

Theorem  \ref{stochbound} and Proposition \ref{k1/k0-random-prop} show that  the convergence rate of
the proposed MsFEM for the stochastic equation (\ref{model-spde})   depends on the energy ratio $E(m)$.  Then we immediately have the following theorem.

\begin{theorem}
Suppose that the assumptions in Proposition \ref{k1/k0-random-prop} hold. If $m$ is large enough such that
$C_Y\big(1-E(m)\big)$ is less than  1, then
\[
||\sqrt{k}\nabla(u_h-u_{J, h})||_{L^2(D)\otimes L^2(\Omega)} \leq C\big(C_Y\big(1-E(m))\big)^{\frac{J+1}{2}}.
\]
\end{theorem}

If the stochastic field  $Y(x, \omega)$ is Gaussian, then its
covariance function  can analytically be extended to the whole complex plane $\mathbb{C}^d$, which is stronger than  piecewise analytic.
The  eigenvalues $\{\lambda_i\}$ associated with Gaussian fields decay very fast. Consequently, $1-E(m)$ will decay very fast as $m$ increases.
Due to central limit theorem, the Gaussian stochastic process is very interesting for applications.

For stochastic simulation, we can use Monte Carlo methods. The main disadvantage of Monte Carlo methods
is slow convergence. To overcome the disadvantage, here  we use  stochastic collocation methods to discretize random parameter space.
The proposed  MsFEM is used to discretize the spatial variable.  Combined with  the new MsFEM,   we
develop modified stochastic collocation methods  to reduce the dimension of the random parameter space. Computational complexity
is addressed for  the modified stochastic collocation method.

\subsection{Stochastic collocation methods and random parameter space reduction} \label{paramreduc}
In this subsection, we use stochastic collocation methods to discretize the random parameter space  and
 show that combining  the proposed  multiscale method with stochastic collocation methods
can reduce the dimension of the random parameter space, which
is very important for simulations in high-dimensional random space.

Let $\{\Theta_0^1,\Theta_0^2,\cdots, \Theta_0^s\}\subset \mathbb{R}^m$ be $s$ collocation points scattered  in
random parameter space associated with an interpolation operator $I_m$.
 Let $v(\Theta_0)\in C(\mathbb{R}^m)$ be a deterministic solution depending
on random parameters $\Theta_0$. Then given a realization $\Theta_0\in \mathbb{R}^m$, the collocation solution $v_m$ is defined
by $v_m(\Theta_0)=I_m v(\Theta_0)$.  We usually use the roots of an orthogonal polynomial (e.g., Hermite polynomial or Chebyshev polynomial)
to find the interpolation points. One can select  different  collocation points and use a different
interpolation operator $I_m$ to obtain different stochastic collocation methods,
for example, full tensor product collocation \cite{bnt07} and Smolyak sparse grid collocation \cite{bnr00}.

Suppose that the Green's function $G(x, y, \Theta_0)$ in (\ref{eq-green}) depends on the random parameter  $\Theta_0\in \mathbb{R}^m$
and  $I_m G(x, y, \Theta_0)$ is the collocation solution for an arbitrary $\Theta_0 \in \mathbb{R}^m$.
For any arbitrary realization $\Theta:=(\Theta_0, \Theta_1)\in \mathbb{R}^m\times \mathbb{R}^{n-m} $,
we define a modified interpolation operator $\tilde{I}_m$  for  $\Pi l$ and $\tx_j$ ($j=0,\cdots, J$).
We define them as follows:

\begin{equation}
\label{interp-pi-l}
\tilde{I}_m [(\Pi l)(x, \Theta_0)]:=\int_K k_0(y, \Theta_0)  \nabla_y I_m G(x,y, \Theta_0)\cdot \nabla_y l dy.
\end{equation}
We note that we can get the value of $k_0(y, \Theta_0)$ and do not employ interpolation $I_m$ for $k_0(y, \Theta_0)$.
Then we similarly  compute $\tilde{I}_m [\tilde{\xi}_0]$ by
\[
\tilde{I}_m [\tilde{\xi}_0(x, \Theta)]:
= \int_K \, k_1(y, \Theta) \nabla_y I_m G(x,y, \Theta_0) \cdot \bigg(\nabla_y \tilde{I}_m [(\Pi l)(y, \Theta_0)] -\nabla_y l(y) \bigg) dy,
\]
and compute $\tilde{I}_m [\tilde{\xi}_k]$, $k=1,2,3,\cdots$ by performing

\begin{equation} \label{greeninterp}
\tilde{I}_m [\tilde{\xi}_j(x, \Theta)]:
=- \int_K \, k_1(y, \Theta) \nabla_y I_m G(x,y, \Theta_0) \cdot \nabla_y \tilde{I}_m [\tx_{j-1}(y,\Theta)] dy.
\end{equation}

By the definitions of $\tilde{I}_m [\Pi l]$ and $\tilde{I}_m[\tx_j]$ ($j=0, \cdots, J$),
we have only used $I_m(G(x, y, \Theta_0))$ to compute $\tilde{I}_m [\Pi l]$ and $\tilde{I}_m[\tx_j]$ ($j=0, \cdots, J$). This interpolation is performed on the relatively low
dimensional parameter space $\mathbb{R}^m$.
Since the new multiscale basis function is defined as
$\phi_J(x, \Theta)=l(x)-\Pi l(x, \Theta_0)+\sum_{j=0}^J {\tx_j(x, \Theta)}$,   the computation
of the interpolation for $\phi_J(x, \Theta)$ only involves the $m$-dimensional interpolation
$I_m(G(x, y, \Theta_0))$.

To address the complexity, we use matrix and inner products to discuss  the computation of the interpolations in (\ref{interp-pi-l})
and (\ref{greeninterp}).
Due to (\ref{pi-l-fe}) and (\ref{interp-pi-l}), it follows that
\begin{equation}
\label{D-Green1}
\tilde{I}_m [\Pi l (x, \Theta_0)]=\bigg( [I_m M_0^{-1}(\Theta_0)] \vec{b}(x)\bigg)^T v_0(\Theta_0).
\end{equation}
By (\ref{tx-fe}) and (\ref{greeninterp}), we have for $j=0, \cdots, J$,
\begin{eqnarray}
\label{D-Green2}
\begin{split}
\tilde{I}_m [\tilde{\xi}_j(x, \Theta)]&=
(-1)^j \bigg([I_m M_0^{-1}(\Theta_0)] \vec{b}(x)\bigg)^T\bigg(M_1(\Theta) [I_m M_0^{-1}(\Theta_0)]\bigg)^j\\
      & \times  \bigg(M_1(\Theta) [I_m M_0^{-1}(\Theta_0)] v_0(\Theta_0)-v_1(\Theta)\bigg).
\end{split}
\end{eqnarray}
By (\ref{D-Green1}) and (\ref{D-Green2}), we  compute  $I_m M_0^{-1}(\Theta_0)$, $v_0(\Theta_0)$, $v_1(\Theta)$ and $M_1(\Theta)$
 to obtain  $\tilde{I}_m [(\Pi l) (x, \Theta_0)$ and $\tilde{I}_m [\tilde{\xi}_j(x, \Theta)]$ ($j=0,\cdots, J$).
The computations of $v_0(\Theta_0)$, $v_1(\Theta_1)$ and $M_1(\Theta)$   are independent each other and
very efficient in parallel.     The dominant computation
lies in $I_m M_0^{-1}(\Theta_0)$ and solely  depends on the dimension $m$
of the random parameter space. We can very efficiently compute $I_m M_0^{-1}(\Theta_0)$
in  parallel as well.

If we use the standard multiscale basis function defined in equation (\ref{MsFE-basis1}),
then the basis function is $\phi:=\phi(x, \Theta)$ for an arbitrary realization $\Theta\in \mathbb{R}^n$.
We interpolate  the basis function $\phi(x, \Theta)$ in the full random space $\mathbb{R}^n$ ($n>m$). If $n$ is large,
the interpolation on $\mathbb{R}^n$ is computationally expensive and prohibitive.

If we use full tensor product collocation and polynomials with degree $q$ for each component of
$\Theta\in \mathbb{R}^n$,  then we have $(q+1)^n$ collocation points for full-space interpolation.
Consequently, we need to compute $(q+1)^n$ multiscale basis equations $\phi$ for each
vertex for the collocation.  However, if we use the technique of the Green's function
for the new multiscale basis functions, then we have $(q+1)^m$  collocation points and
compute only $(q+1)^m$ Green's functions (or Green's matrix $M_0$) to generate the proposed multiscale basis functions.  Consequently,
the ratio $\alpha_{ftc}$  between  the number of collocation points in the proposed multiscale basis function  and in
the standard multiscale basis function is
\[
 \alpha_{ftc}= (q+1)^{m-n}.
 \]
For example, if $n=20$, $m=10$ and $q=2$, then the $\alpha_{ftc}=3^{-10}$. This means that the computational effort
 for using the proposed  multiscale method may be considerably decreased compared to the standard multiscale method
 when the full tensor product collocation is employed.

Let $H(n+L,  n)$ denote the interpolation nodes for Smolyak sparse grid collocation \cite{bnr00} at dimension $n$ and interpolation level $L$. Although Smolyak sparse grid
collocation requires much fewer nodes than the full tensor product collocation to achieve the similar accuracy,
the number of nodes $H(n+L,  n)$ increases  very quickly as $n$ increases.
The ratio $\alpha_{ftc}$  between  the number of collocation points in the proposed multiscale basis function  and in
the standard multiscale basis function is given by
\[
 \alpha_{sgc}= \frac{H(m+L, m)}{H(n+L, n)}\approx (\frac{m}{n})^L  \quad \text{for $m\gg 1$, $n\gg 1$},
 \]
 where we have used the fact $H(n+L, n)\approx \frac{2^L}{L!} n^L$ for $n\gg 1$ \cite{bnr00}.   For example, if we take interpolation level $L=2$, $m=10$ and $n=20$,  then $\alpha_{sgc}=\frac{221}{841}\approx {1\over 4}$. This means that the computation
 time of the proposed  multiscale method is almost ${1\over 4}$ of the standard multiscale method in parallel setting.

 Smolyak sparse grid collocation is known to have the same asymptotic accuracy as full tensor product collocation, while requiring
many  fewer interpolation points as the parameter dimension increases. We will use Smolyak sparse grid collocation for the numerical tests.
The stochastic approximation of the Smolyak sparse grid collocation method, $\|v-I_m v\|$,  depends on the total
number of sparse grid collocation nodes and the dimension $m$ of the random parameter space.
The convergence analysis in \cite{ntw08} implies that the convergence  of Smolyak sparse grid collocation  with respect to the number of Smolyak nodes is exponential, but depends on the parameter dimension $m$.  If $m>>1$, then the exponential convergence rate behaves algebraically.

Using the modified stochastic collocation method described in (\ref{interp-pi-l}) and (\ref{greeninterp}), we define
the corresponding collocation basis function $\tilde{\phi}_J(x, \Theta)=l(x)-\tilde I_m [\Pi l(x, \Theta_0)]+ \sum_{j=0}^J \tilde{I}_m [\tx_{j} (x, \Theta)]$.
Let $\tilde{u}_{J, h}$ be the collocation solution using the basis $\tilde{\phi}_J(x, \Theta)$.  Then the total error
$||\sqrt{k}\nabla(u_h-\tilde {u}_{J, h})||_{L^2(D)\otimes L^2(\Omega)}$ includes two parts:
splitting error $||\sqrt{k}\nabla(u_h-u_{J, h})||_{L^2(D)\otimes L^2(\Omega)}$ and collocation error $||\sqrt{k}\nabla(u_{J, h}-\tilde{u}_{J, h})||_{L^2(D)\otimes L^2(\Omega)}$.
It can be formally expressed by
\[
||\sqrt{k}\nabla(u_h-\tilde {u}_{J, h})||_{L^2(D)\otimes L^2(\Omega)}\leq e_{spl} +e_{col},
\]
where $e_{spl}=O\bigg(\big(1-E(m)\big)^{\frac{J+1}{2}}\bigg)$ and $e_{col}=e_{col}(J, m, L)$.
There exists a trade-off between the splitting error $e_{spl}$ and the collocation error $e_{col}$.
The numerical results in Section $5$ illustrate the finding.

Using the proposed stochastic collocation methods, computation  of the MsFEM basis function $\tilde{\phi}_J$  is summarized in Algorithm \ref{alg:col-msfem}.

\begin{algorithm}[h]\caption{Computing  MsFEM basis functions $\tilde{\phi}_J$ using the parameter reduction collocation method}\label{alg:col-msfem}
\begin{enumerate}
\item Choose collocation samples $\{\Theta_0^i \}_{i=1}^s\subset \mathbb{R}^m$;
\item Assemble Green's matrices  $M_0(\Theta_0^i)$ ($i=1, \cdots, s$) for all collocation samples independently;
\item Given an arbitrary realization $\Theta:=(\Theta_0, \Theta_1)\in \mathbb{R}^m\times \mathbb{R}^{n-m}$,
assemble vectors $v_0(\Theta_0)$ and $v_1(\Theta_1)$ and matrix $M_1(\Theta)$ independently;
\item Compute $I_m M_0^{-1}(\Theta_0)$ indpendently;
\item Compute $\tilde{I}_m [(\Pi l) (x, \Theta_0)$ by (\ref{D-Green1}) and $\tilde{I}_m [\tilde{\xi}_j(x, \Theta)]$ ($j=0,\cdots, J$) by
(\ref{D-Green2}) independently;
\item Construct the interpolated basis function $\tilde{\phi}_J(x, \Theta):=l(x)-\tilde{I}_m [(\Pi l) (x, \Theta_0)+\sum_{j=0}^J \tilde{I}_m [\tilde{\xi}_j(x, \Theta)]$.
\end{enumerate}
\end{algorithm}


 \section{Numerical Results}
 In this section we offer a number of representative numerical results to verify the analysis and evaluate the performance of the proposed method.

 \subsection{Deterministic basis function} \label{detbasis}
 To begin, we consider the analysis offered in Sect.~\ref{converge}. In particular, we are initially interested in verifying the convergence properties of a single, deterministic basis function as described in Lemma.~\ref{singlethm} (or Prop. ~\ref{prop3.2}). In this subsection we consider two distinct cases of coefficient examples. We first consider a coefficient generated by a Karhunen-Lo\`{e}ve expansion and a coefficient constructed from a log-normal distribution. See Fig.~\ref{logperm} for illustrations of each coefficient plotted on the log scale.

 \begin{figure}
\centering
   \includegraphics[width=0.6\textwidth]{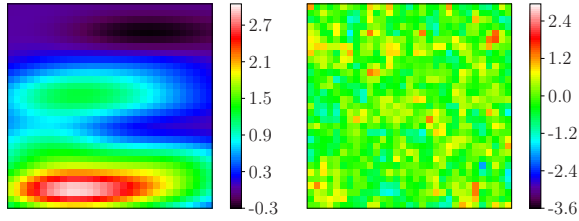}
   \caption{KLE (left) and log-normal (right) coefficients posed on a $30 \times 30$ mesh (log scale) }
   \label{logperm}
\end{figure}
To generate the first test coefficient on an arbitrary coarse element $K$ we employ the KLE expansion from Eq.~\eqref{KLE}. In our case we use the correlation function
\begin{equation} \label{corrfunct}
\text{cov}[Y](\mathbf{x}_1, \mathbf{x}_2) := R(x_1, y_1; x_2, y_2) = \sigma^2 \text{exp} \left( - \frac{|x_1 - x_2|^2}{2 l_x} - \frac{|y_1 - y_2|^2}{2 l_y} \right),
\end{equation}
where $\sigma^2$ is the variance, and $l_x, l_y$ denote the correlation lengths in the $x-$ and $y-$directions, respectively. We consider an elliptic coefficient which is generated on a $30 \times 30$ mesh. For the variance we use $\sigma^2 =  2.25$ and for the correlation lengths we use $l_x = 0.2$ and $l_y = 0.05$. For all examples we truncate the original KL expansion at $n=20$ terms to obtain the full coefficient $k = k_0 + k_1$. Then, in order to split the coefficient accordingly, we may choose a variety of $m$ values and employ Eq.~\eqref{split-random-k} for the splitting. For example, we may use $m=5$ terms to obtain $k_0$, and $k_1=k-k_0$. See Fig.~\ref{permsplit} for a representative example of a KLE coefficient splitting.
\begin{figure}
\centering
   \includegraphics[width=0.85\textwidth]{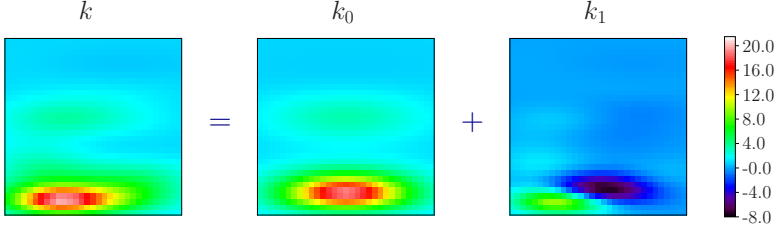}
   \caption{KLE coefficient decomposition posed on a $30 \times 30$ mesh; $n=20$, $m = 5$}
   \label{permsplit}
\end{figure}
As the initial analysis is built in a deterministic setting, we use the same, fixed $\theta_i$ $(i=1, \ldots, n)$ in \eqref{split-random-k} for all related examples. To begin, we recall the error estimate
\begin{equation} \label{errest}
|||\xi-\xi_J|||_K
 \leq 2  \|\frac{k_1}{\sqrt{kk_0}}\|_{L^{\infty}(K)} \eta_K^{J+1} \|\sqrt{k_0} \nabla l\|_{0, K}
\end{equation}
and set $\mathcal{B}^\xi := 2  \|\frac{k_1}{\sqrt{kk_0}}\|_{L^{\infty}(K)} \eta_K^{J+1} \|\sqrt{k_0} \nabla l\|_{0, K}$ from the last lines in the proof of Lemma.~\ref{singlethm}. In particular, we recall that when $\eta_K=\|{k_1\over k_0}\|_{L^{\infty}(K)} < 1$, convergence of the basis function sequence $\{ \phi_J \}_{J=0}^{\infty}$ is expected from the analysis. To test the error bound in Eq.~\eqref{errest} we consider a variety of field splitting configurations. In Fig.~\ref{boundplot} we illustrate two cases of splitting where $\eta_K=\|{k_1\over k_0}\|_{L^{\infty}(K)} < 1$. These examples result from the cases where $m=15$ and $m=17$ terms are used in the KLE splitting. First, it is important to note that the bounds presented in the analysis are clearly represented in the figure. In particular, for either case we see that the energy norm of the error is always bounded above by the theoretical estimate provided in Eq.~\eqref{errest}. We also note that as $J$ (the number of terms in the approximate basis function sequence)  increases the error and associated bounds rapidly decrease. This behavior is expected as the term $\eta_K^{J+1}$ quickly decreases as $J$ increases. We also point out that a smaller value of $\eta_K$ yields a tighter bound.
\begin{figure}
\centering
   \includegraphics[width=0.9\textwidth]{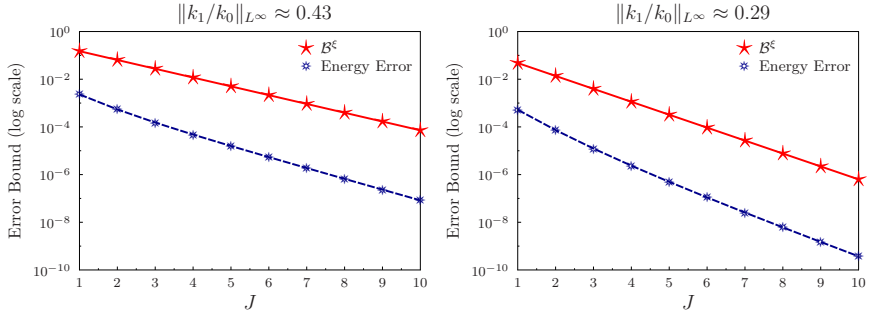}
   \caption{Energy error and error bound computations for the KLE coefficient; $m=15$ terms (left), $m=17$ terms (right) }
   \label{boundplot}
\end{figure}
To conclude this coefficient example, we offer an illustration of the actual basis functions which are obtained through the proposed computational method. Fig.~\ref{phixiplot} includes the benchmark basis function $\phi$ as well as $\phi_J$ $(J=0,1,2)$ from the sequence $\{ \phi_J \}$. We also plot the benchmark perturbation $\phi - l$ and $\phi_J - l$ for $J=0,1,2$. All plots were obtained from the case where $m=5$ terms were used in the KLE splitting. We note that for a relatively pronounced splitting, we see a noticeable convergence trend.
\begin{figure}
\centering
   \includegraphics[width=1.0\textwidth]{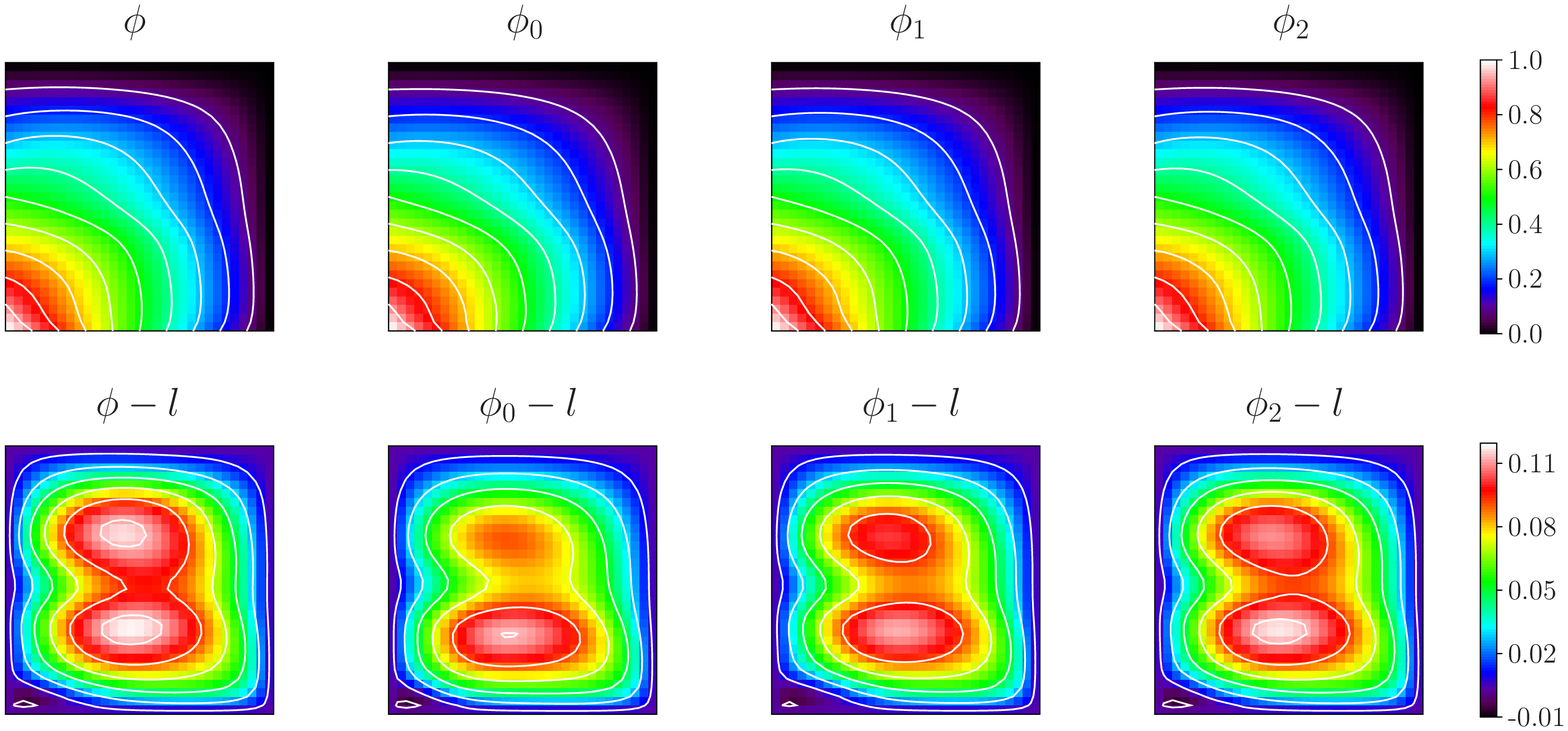}
   \caption{Convergence illustration for a basis function $\phi$ and the corresponding perturbation $\phi - l$ for the KLE coefficient; $m=5$ }
   \label{phixiplot}
\end{figure}

To generate a second test coefficient on an arbitrary coarse element $K$ we assume that the full coefficient $k$ follows a log-normal distribution. That is, we assume that $\text{ln}\left( k(x) \right) = Y(x, \omega)$, where $Y(x, \omega)$ is a normal random variable with zero mean and variance one. We also assume that $\text{ln}\left( k_0(x) \right) = s_c Y(x, \omega)$, where $s_c$ is the strength factor which will determine the final splitting $k = k_0 + k_1$. In particular, we set

\begin{equation} \label{normalsplit}
k_1 = k - k_0 = \text{exp}(Y) - \text{exp}(s_c Y)
\end{equation}
to create the coefficient decomposition. For the following examples, we choose various values of $s_c$ within the interval $0.4 < s_c < 0.96$. See Fig.~\ref{permsplit_norm} for an example of the coefficient splitting in Eq.~\eqref{normalsplit} for the case when $s_c = 0.4$. We note that the $k_0$ portion of the decomposition is much less heterogeneous than $k_1$.

\begin{figure}
\centering
   \includegraphics[width=0.85\textwidth]{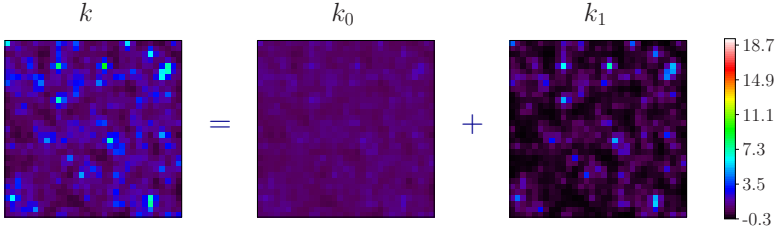}
   \caption{Log-normal coefficient decomposition posed on a $30 \times 30$ mesh; $s_c$ = 0.4}
   \label{permsplit_norm}
\end{figure}

To validate the convergence properties outlined in Thm.~\ref{singlethm}, we consider two approaches and recall the estimate from Eq.~\eqref{errest}. In Fig.~\ref{boundplot_normal} we illustrate two representative plots of splitting where $\eta_K=\|{k_1\over k_0}\|_{L^{\infty}(K)} < 1$. These examples result from the cases where strength factors of $s_c = 0.90$ and $s_c = 0.94$ are used in the log-normal coefficient splitting. As before, we note that the bounds presented in the analysis are clearly represented in the figure. For either case we see that the energy norm of the error is always bounded above by the theoretical estimate provided, and that as $J$ increases, the error and associated bounds decrease rapidly. Here, we are also interested in the rate of convergence offered in Eq.~\eqref{errest}. To address this, we fix a value of $J$ and plot $\eta_K$ vs. $|||\phi-\phi_J|||_K$ on the log scale. Fig~\ref{orderplot_normal} illustrates the log plots as well as the slopes obtained from a linear trend line. In this case we obtain slopes which are close to the value of $J+2$. In particular, for $J=0$ we obtain a slope of $1.8$, for $J=2$ we obtain a slope of $3.7$ and for $J=4$ we obtain a slope of $5.6$. These results are consistent with the convergence rate results from Prop.~\ref{prop3.2}. In particular, we see from the plots in Fig.~\ref{orderplot_normal} that exponents of $J + 2 - \delta$ are recovered for all values of $J$.  Because the estimate in Prop.~\ref{prop3.2} depends on a constant (maybe not sharp),
 there exist a slight difference $\delta$ of the convergence rate in the numerical results compared to the estimate in in Prop.~\ref{prop3.2}.

\begin{figure}
\centering
   \includegraphics[width=0.9\textwidth]{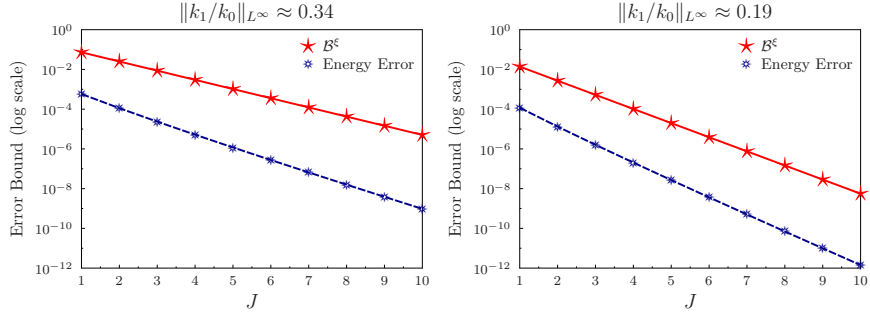}
   \caption{Energy error and error bound computations for the log-normal coefficient; $s_c = 0.90$ (left), $s_c = 0.94$ (right) }
   \label{boundplot_normal}
\end{figure}
\begin{figure}
\centering
   \includegraphics[width=1.0\textwidth]{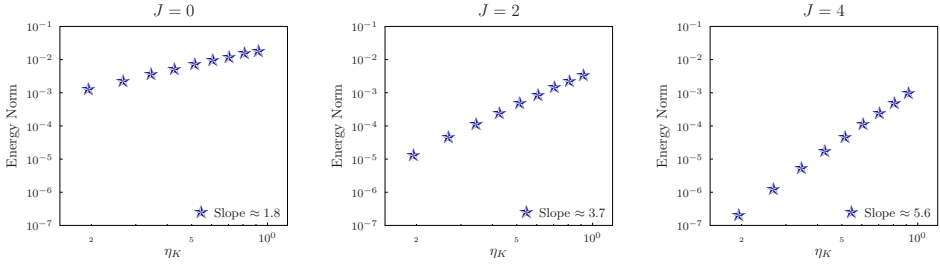}
   \caption{Log plot of $\eta_K$ vs. $|||\phi-\phi_J|||_K$ for the log-normal coefficient; $J = 0$ (left), $J = 2$ (center), $J = 4$ (right) }
   \label{orderplot_normal}
\end{figure}
To conclude these coefficient examples, we offer a representative illustration of the actual basis functions which are obtained through the proposed computational method. Fig.~\ref{phixiplot_normal} includes the benchmark basis function $\phi$ as well as $\phi_J$ $(J=0,4,8)$ from the sequence $\{ \phi_J \}$. We also plot the benchmark perturbation values $\phi - l$ and $\phi_J - l$ for $J=0,4,8$. All plots were obtained from the case where a strength factor of $s_c = 0.4$ is used in the coefficient splitting. We note that this is a rather extreme splitting where $k_0$ is much less heterogenous compared to $k_1$ (refer back to Fig.~\ref{permsplit_norm}) for the original log-normal coefficient configuration. The figure illustrates that successive approximations offer an accurate counterpart to the benchmark basis and perturbation values.

\begin{figure}
\centering
   \includegraphics[width=1.0\textwidth]{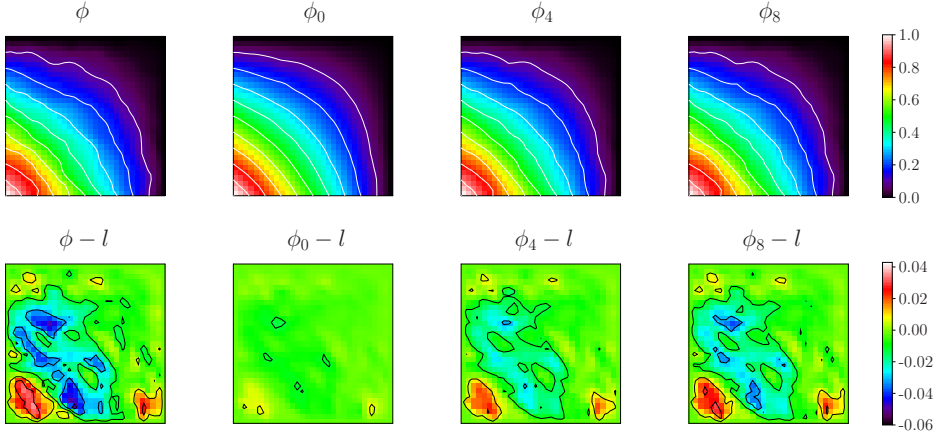}
   \caption{Convergence illustration for a basis function $\phi$ and the corresponding perturbation $\phi - l$ for the log-normal coefficient; $s_c=0.4$ }
   \label{phixiplot_normal}
\end{figure}

 \subsection{Deterministic elliptic solution}
 In this subsection we assess the convergence of the respective solutions of Eq.~\eqref{model-eq}. In particular, we are interested in comparisons between the standard MsFEM solution $u_h \in V_h$, and the proposed MsFEM solution $u_{J,h} \in V_{J,h}$. We first recall that Thm.~\ref{det-thm} suggests convergence of $u_{J,h}$ to $u_{h}$ as $J \rightarrow \infty$. In order to verify this theoretical result we test two separate permeability configurations. In particular, we use a KL expansion (with a set of fixed random parameters), and a single realization of a log-normal field analogous to those in Subsection \ref{detbasis}. However, the fine scale fields are now posed on at least a $120 \times 120$ fine mesh. See Fig.~\ref{logperm_fine} for a log-plot of each respective permeabililty field. Until otherwise noted, all MsFEM solutions in this subsection are obtained by using a $12 \times 12$ coarse mesh.

 \begin{figure}
\centering
   \includegraphics[width=0.6\textwidth]{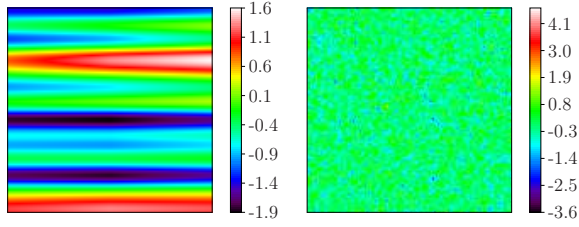}
   \caption{Fine scale KLE (left) and log-normal (right) coefficients posed on a $120 \times 120$ mesh (log scale) }
   \label{logperm_fine}
\end{figure}
First, we consider a KLE coefficient which is posed on $120 \times 120$, $240 \times 240$, and $360 \times 360$ fine meshes. For the variance we use $\sigma^2 =  2.25$ and for the correlation lengths we use $l_x = 0.7$ and $l_y = 0.04$, respectively (see Eq.~\eqref{corrfunct}). For all examples we truncate the original KL expansion at $n=20$ terms to obtain the full coefficient $k = k_0 + k_1$. In order to test the convergence properties in Thm.~\ref{det-thm} we are interested in computing a variety of proposed MsFEM solutions $u_{J,h}$ (for $J=1,\ldots$). Then, we compute the energy norm $||| u_h - u_{J,h} |||$ and the associated bound $\mathcal{B}^u := (\tilde{c})^{1/2} \big(\eta^{J+1}+\frac{\eta^{J+1}}{1-\eta^{J+1}}\big )^{1/2} |||u|||$. We note that for $||| u |||$, we compute the energy norm of the standard FEM elliptic solution. Fig.~\ref{fineboundplot_kl} illustrates the energy norm values and bounds for increasing $J$ values, and two different KLE splitting configurations ($m=16$ and $m=18$). We note that, as before, a smaller $\eta$ value yields a tighter bound. In addition, we see the pronounced convergence of the proposed MsFEM solution. Leaving the coarse mesh intact, Fig.~\ref{fineboundplot_kl} also serves to illustrate the effect that different local meshes have on the construction of the operators $M_0$ and $M_1$ from Thm.~\ref{thm-v-M}. In essence, the fine meshes give three successive local mesh refinements (from $10 \times 10$ to $20 \times 20$ to $30 \times 30$) on which the local operators will be constructed. As shown in the figure, there is a negligible difference in the respective errors, thus illustrating that the fine scale on which the operators are computed does not have a significant effect on the error resulting from the proposed method. Furthermore, the results show that the original $120 \times 120$ mesh is sufficiently fine to resolve the necessary scales in the KL expansion. To further illustrate the convergence of the proposed MsFEM approach, we offer various elliptic solution plots in Fig.~\ref{ellipcompare_kl}. In addition to the solution plots we offer the relative errors $|||u_h - u_{J,h}||| / ||| u_h |||$. As $J$ increases, we see that the relative error steadily decreases from $3.6 \%$ to $0.8 \%$.

\begin{figure}
\centering
   \includegraphics[width=0.9\textwidth]{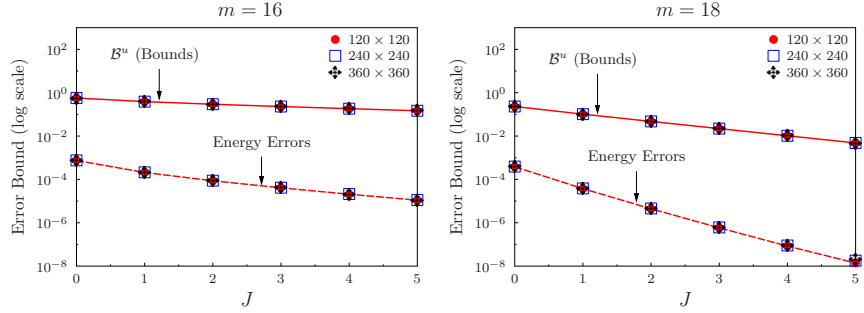}
   \caption{Energy error and error bound computations for the fine scale KLE coefficient; $m=16$ terms (left), $m=18$ terms (right); $120 \times 120$, $240 \times 240$, and $360 \times 360$ fine
                  meshes}
   \label{fineboundplot_kl}
\end{figure}
\begin{figure}
\centering
   \includegraphics[width=1.0\textwidth]{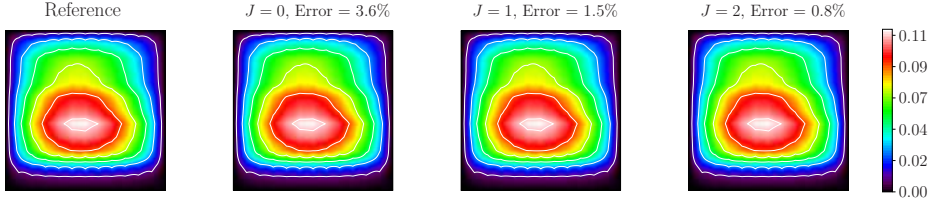}
   \caption{Convergence illustration for the reference MsFEM elliptic solution $u_h$ (labeled `Reference'), and the corresponding new MsFEM elliptic solutions $u_{J,h}$ for
                 $J=0,1,2$; KLE coefficient, $m=12$ }
   \label{ellipcompare_kl}
\end{figure}

We also consider the log-normal coefficient to test the convergence properties of the proposed approach. The energy norm and bound plots may be seen in Fig.~\ref{fineboundplot_normal}, where strength values of $s_c = 0.88$ and $s_c = 0.9$ are used. As expected, we see that the solutions converge and that a smaller $\eta$ value yields a tighter bound. For this case, we also offer an illustration of various solution plots along with the relative errors values. See Fig.~\ref{ellipcompare_normal} for the convergence illustration. We note that as $J$, the relative error decreases from $1.0 \%$ to $0.2 \%$. In this case we see that the initial error is less than its counterpart from the KLE coefficient. This is not unexpected since the KLE expansion represents a stronger form of heterogeneity (refer back to Fig.~\ref{logperm_fine}).

\begin{figure}
\centering
   \includegraphics[width=0.9\textwidth]{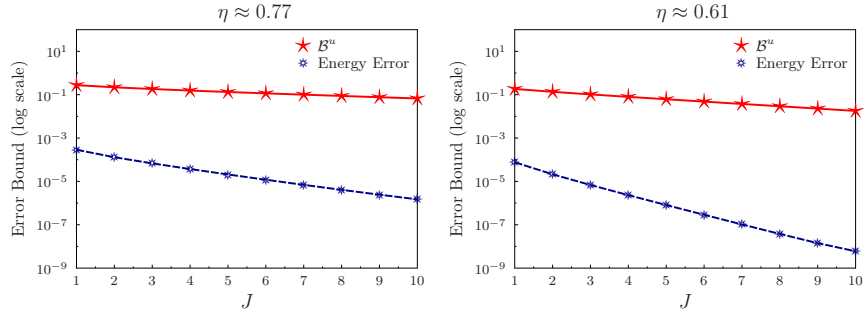}
   \caption{Energy norm and error bound computations for the fine scale log-normal coefficient; $s_c=0.88$ (left), $s_c=0.9$ (right)}
   \label{fineboundplot_normal}
\end{figure}
\begin{figure}
\centering
   \includegraphics[width=1.0\textwidth]{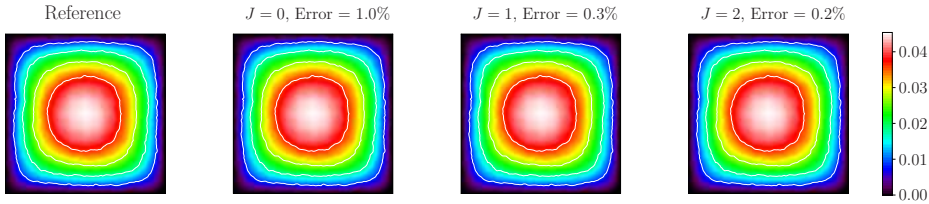}
   \caption{Convergence illustration for the reference MsFEM elliptic solution $u_h$ (labeled `Reference'), and the corresponding new MsFEM elliptic solutions $u_{J,h}$ for
                 $J=0,1,2$; log-normal coefficient, $s_c = 0.86$ }
   \label{ellipcompare_normal}
\end{figure}

In addition to the results presented above, we also consider a variety of coarse mesh configurations for comparison. To recall, all previous examples in this subsection were obtained from a fixed $12 \times 12$ coarse mesh. However, it is also fitting to illustrate the effects of different mesh configurations on the solution error. In particular, Fig.~\ref{errorplot_varycoarse} contains the errors quantities $||| u - u_{J,h} |||$ (left) and $||| u_h - u_{J,h} |||$ (right) for a variety of mesh configurations and $J$ values. The errors are obtained from the same KLE coefficient data with $m = 16$. We note that the left set of errors (the errors between the standard FEM solution and the proposed MsFEM solution) are essentially constant regardless of the value of $J$. In other words, the dominant source of error is clearly from the multiscale solution method, and the error between standard MsFEM and the proposed method is negligible. In addition, we see that a refinement of the coarse mesh yields a smaller error, which is to be expected from a multiscale solution technique. Of particular interest are the right set of errors $||| u_h - u_{J,h} |||$, which are computed using the same coarse mesh configurations. Most importantly, we see that a successive refinement of the coarse mesh does not significantly affect the error between standard MsFEM and the proposed method. We note that a refined coarse mesh does lead to a slight decrease in the error, however, it is clear from Fig.~\ref{errorplot_varycoarse} that the proposed method is not sensitive with respect to the coarse mesh configuration. In particular, it is evident that $J$ is the parameter which dictates the respective errors.

\begin{figure}
\centering
   \includegraphics[width=0.9\textwidth]{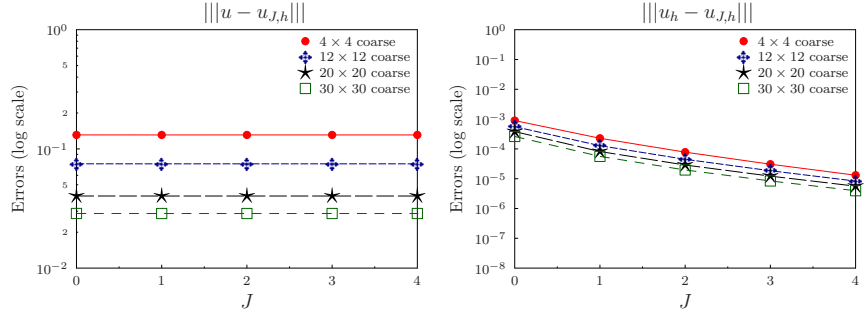}
   \caption{Energy error computations for the fine scale KLE coefficient; $m=16$; $4\times4, 12\times12, 20\times20$, and $30\times30$ coarse meshes}
   \label{errorplot_varycoarse}
\end{figure}

 \subsection{Stochastic elliptic solution using the Monte Carlo method} \label{stochresults}
In this subsection we address the stochastic problem described in Sect.~\ref{stochanalysis}. To begin, we are first interested in testing the stochastic bounds which are proved in Thm.~\ref{stochbound}. We remark that this stochastic result is analogous to the deterministic bounds offered in Thm.~\ref{det-thm}. In order to compute the left hand side of the inequality in Thm.~\ref{stochbound} we use the Monte Carlo approximation

\begin{equation}
||\sqrt{k}\nabla(u_h-u_{J, h})||_{L^2(D)\otimes L^2(\Omega)} \approx \sum_{s=1}^{N}  ||\sqrt{k}\nabla(u_h^s-u_{J, h}^s)||_{L^2(D)} / N = \sum_{s=1}^{N}  ||| u_h^s-u_{J, h}^s||| / N,
\end{equation}
where the index $s$ denotes a fixed sample value. In particular, we generate a stochastic field, compute the corresponding energy norms for $s=1, \ldots, N$, and obtain an average over the samples. For the right hand side of the inequality $\widetilde{\mathcal{B}}^u := (\tilde{c})^{1/2} \big(\eta^{J+1}+\frac{\eta^{J+1}}{1-\eta^{J+1}}\big )^{1/2} ||\sqrt{k}\nabla u||_{L^2(D)\otimes L^2(\Omega)}$, we use the same type of Monte Carlo approximation for the stochastic integrals and we use the fully resolved FEM solution for $u$. In order to verify the bounds, we focus on a stochastic field which is generated from a truncated KL expansion. For the results in this subsection we employ correlation lengths of $l_x = l_y = 0.1$ and a variance of $\sigma^2 = 1.0$ for the field construction. The series expansion is truncated at $n=20$ terms, and the coefficient is posed on a $64 \times 64$ fine mesh. We assume that the random coefficients $\theta_i$ ($i=1, \ldots, 20$) represent a $20$ dimensional vector in the hypercube $[-1,1]^{20}$. In other words, for sampling we draw $20$ i.i.d. uniform random variables from the interval $[-1,1]$. See Fig.~\ref{permsample} for a typical coefficient sample in this context. For all stochastic computations $N = 200$ samples are used, and for all MsFEM computations we use a $16 \times 16$ coarse mesh.

\begin{figure}
\centering
   \includegraphics[width=0.4\textwidth]{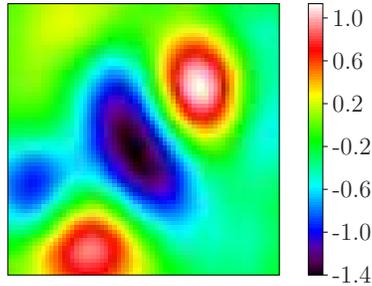}
   \caption{Typical elliptic coefficient sample from the KL expansion posed on a $64 \times 64$ mesh (log scale) }
   \label{permsample}
\end{figure}
To verify the bounds in Thm.~\ref{stochbound} we offer two representative plots in Fig.~\ref{stochbounds}. The left hand side was obtained by keeping $m=14$ terms in the original $20$ term expansion, and the right hand side was obtained by keeping $m = 18$ terms in the original expansion. We note that more terms in the KL expansion (resulting in a smaller average $\eta$ value) yields a tighter bound, and that the expected energy norms deplete rapidly. These results are consistent with those obtained from the deterministic fields.

\begin{figure}
\centering
   \includegraphics[width=0.9\textwidth]{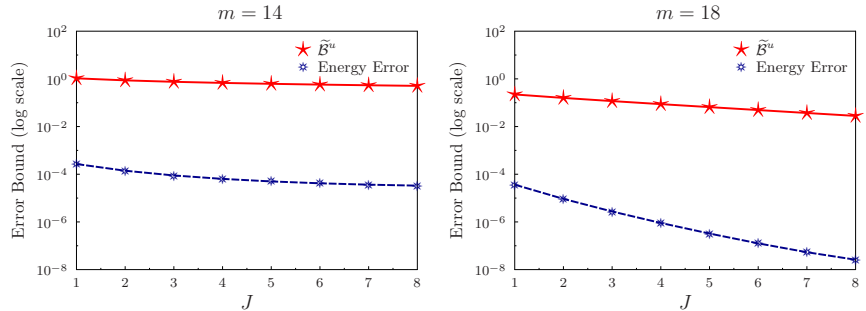}
   \caption{Expected energy norm and error bound computations for the KLE coefficient; $m = 14$ terms (left), $m = 18$ terms (right) }
   \label{stochbounds}
\end{figure}
To further illustrate the behavior of the stochastic problem we recall the energy ratio $E(m)=\frac{\sum_{i=1}^m \sqrt{\lambda_i}}{\sum_i^n \sqrt{\lambda_i}}$, where $\lambda_i$ are the eigenvalues from the KL expansion. We note that as $m$ increases this value is expected to quickly increase (in other words, $1-E(m)$ will quickly decrease). As more terms in the KL expansion typically yield smaller errors from the series approximation, we expect consistent behavior between $1-E(m)$ vs. the mean and variance quantities of the energy norm of the error, $||| u_h - u_{J,h} |||$. Fig.~\ref{statplot} illustrates the relationship between the statistics of the energy norm quantities and the energy ratio. For both plots we use $J=2$ for the basis function approxmations. For both the energy norm mean and variance comparisons, we note that as $1 - E(m)$ increases, the respective statistical values also increase. In particular, we see that less terms in the KL expansion yield errors that grow algebraically with respect to a decreasing energy ratio.

\begin{figure}
\centering
   \includegraphics[width=0.9\textwidth]{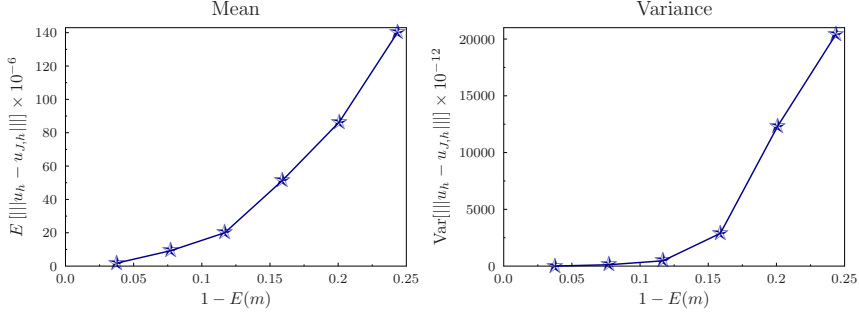}
   \caption{ $1-E(m)$ vs. the mean of $||| \cdot |||$ (left), and $1-E(m)$ vs. the variance of $||| \cdot |||$ (right); $J = 2$, Monte Carlo }
   \label{statplot}
\end{figure}
To finish this subsection we offer statistical comparisons obtained from the reference (standard) MsFEM solution $u_h$ and the proposed  MsFEM solution $u_{J,h}$. See Fig.~\ref{statcompare} for a comparison between the mean and variance of the respective elliptic solutions. We note that any differences are nearly undetectable, further verifying the accuracy of the proposed method.

\begin{figure}
\centering
   \hspace*{-0.8cm} \includegraphics[width=0.8\textwidth]{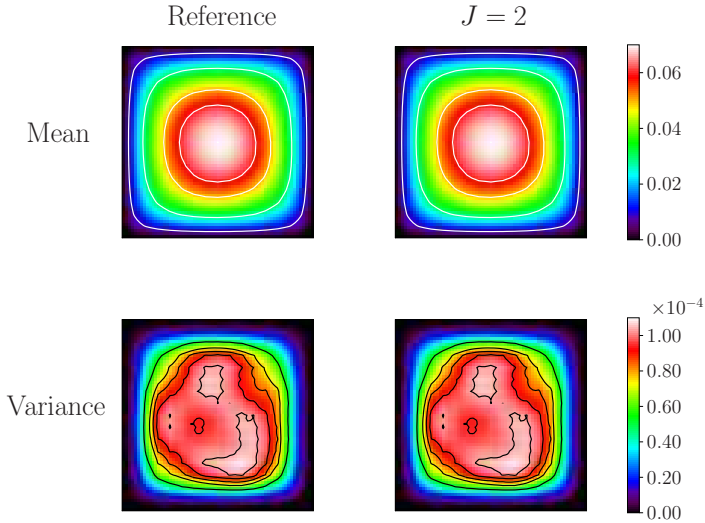}
   \caption{ Statistical comparisons between standard MsFEM and the proposed MsFEM approach }
   \label{statcompare}
\end{figure}

 \subsection{Stochastic elliptic solution using the parameter reduction collocation}
In this subsection we address the alternative parameter reduction collocation approach as described in Subsection~\ref{paramreduc}. In particular, we are interested in testing the accuracy of Monte Carlo sampling with standard MsFEM, versus random parameter reduction collocation with the proposed  MsFEM approach using the Green's kernel. Throughout this section we recall that the total error can be decomposed into two main components. Namely, the error can be decomposed into the splitting error $e_{spl}$, and the collocation error $e_{col} = e_{col}(J,m,L)$. In order to assess the significance of the error contributions, we thoroughly test a variety of scenarios resulting from values of $J$ (number of terms in the series expansion), $m$ (splitting configuration), and $L$ (collocation level). In a parallel setting the computational cost is decreased when the parameter reduction approach is implemented, and demonstrated accuracy of the technique would solidify it as a suitable sampling alternative. Throughout this subsection we use the same KLE configurations as in Subsection~\ref{stochresults}.

To motivate further disussion, we first offer Table~\ref{error} which compares the Monte Carlo sampling approach with the new parameter reduction sampling approach. The values in the table result from keeping $m = 16$ terms for construction of $k_0$ combined with interpolation levels $L=1$, $L=2$, and $L=3$. In particular, we tabulate the values $\frac{ ||| u_h - \tilde{u}_{J,h} ||| }{ ||| u_h ||| } \times 100 \%$ for $5$ fixed $\Theta$ (random parameter) values, where $u_h$ denotes the standard MsFEM solution with Monte Carlo sampling, and $\tilde{u}_{J,h}$ denotes the proposed MsFEM solution with the parameter reduction approach. As we can see from Table~\ref{error}, the parameter reduction approach closely recovers each individual sample of the elliptic solution, and a higher interpolation level leads to smaller errors. In all cases we note that the relative errors do not exceed $1 \%$.

\begin{table}[H]
\begin{center}
\caption{Comparison of elliptic solution difference}
\label{error}
\begin{tabular} [htbp] {|c|c|c|c|}  \hline
\multirow{2}{4.7em}{\hspace*{0.cm}{$\Theta$ Sample}} &
\multicolumn{3}{c|}{Relative Errors $(\%)$}   \\ \cline{2-4}
       & Level 1 & Level 2  & Level 3 \\ \hline
$1$ & 0.64 & 0.12  & 0.05         \\ \hline
$2$ & 0.30 & 0.11   &  0.10       \\ \hline
$3$ & 0.33 & 0.18    &  0.17      \\ \hline
$4$ & 0.67 & 0.19  &    0.08      \\ \hline
$5$ & 0.25 & 0.13   &   0.12       \\ \hline
\end{tabular}
\end{center}
\end{table}

Although the individual sample results are promising, we are most interested in the statistical behavior of the respective solutions. In Fig.~\ref{collstatplot} we offer analogous results to those found in the previous subsection. Namely, we plot $1-E(m)$ vs. the mean and variance quantities of the relative errors, $||| u_h - \tilde{u}_{J,h} ||| / ||| u_h |||$ and $||| u_h - u_{J,h} ||| / ||| u_h |||$, where we recall that $u_{J,h}$ denotes the standard MsFEM solution technique combined with Monte Carlo sampling. In essence, we use these expected errors as benchmarks for comparison with the proposed method. More specifically, we are interested in comparing standard MsFEM with Monte Carlo sampling against the proposed MsFEM approach with parameter reduction collocation. We first note that an increase in interpolation level clearly yields a decrease in the relative error values. The level 2 interpolation errors nearly match the Monte Carlo results, and this slight discrepancy may be viewed as a trade of the increased efficiency of the proposed method. In other words, for a moderate collocation level it is natural to expect some minimal error contributions from collocation. However, the level 3 interpolation yields errors that are nearly identical to the Monte Carlo results. This is due to the fact that the collocation error is essentially dimished, and the only remaining error is due to the splitting. This will be discussed in further detail below. In Fig.~\ref{meanvarL3} we single out the $L=3$ results and plot them with the Monte Carlo results. We again note that the discrepancies are neglible, and point out the similarities in the statistical behavior which is illustrated in Fig.~\ref{statplot}. In particular, we again encounter an increase of the mean and variance quantities with respect $1-E(m)$ which is solely due to the splitting configuration.

\begin{figure}
\centering
   \includegraphics[width=1.0\textwidth]{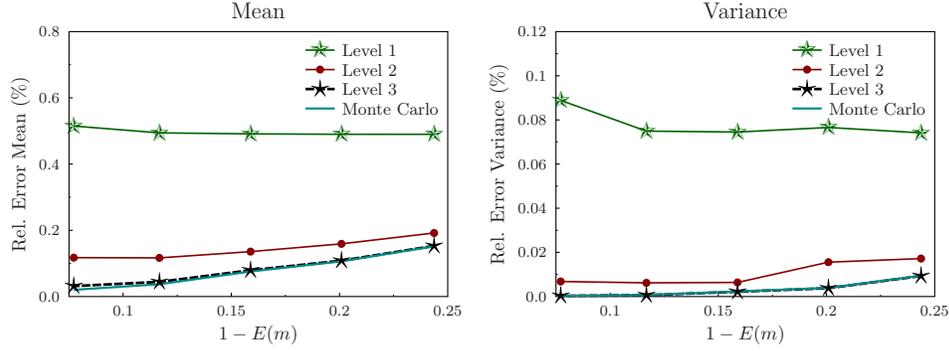}
   \caption{ $1-E(m)$ vs. the expected value (left), and variance (right) of the relative error quantities; $J=1$, $L=1, 2, 3$ and Monte Carlo }
   \label{collstatplot}
\end{figure}

\begin{figure}
\centering
   \includegraphics[width=0.9\textwidth]{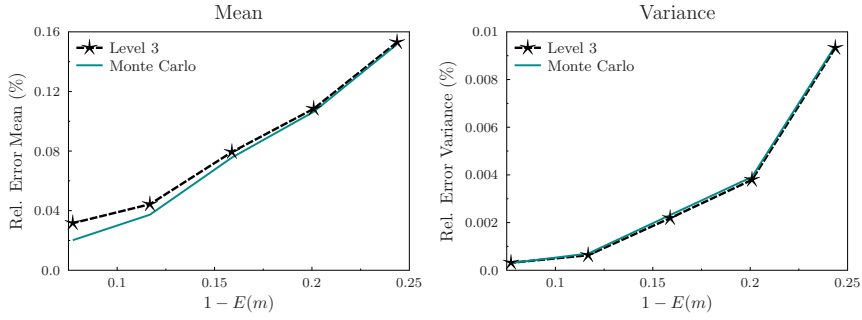}
   \caption{ $1-E(m)$ vs. the expected value (left), and variance (right) of the relative error quantities; $J=1$, $L=3$ and Monte Carlo }
   \label{meanvarL3}
\end{figure}
In addition to the high level of accuracy we obtain for a larger interpolation level, we emphasize that {\it all} plots from Fig.~\ref{collstatplot} illustrate relative errors that do not exceed $1 \%$. Even though the level 1 results do not closely match the Monte Carlo results, a neglible error of $< 1 \%$ may be completely acceptable for many applications. In particular, the small vertical scale should be duely noted. From the $L=1$ results we conclude that the proposed sampling method is not sensitive with respect to $E(m)$, and fewer terms may be kept in the KL expansion without a significant loss in accuracy. This is due to the fact that $e_{col}$ is dominant for the basic collocation level. These results can be viewed as a potential limitation and advantage of the approach if a level 1 interpolant is used. It may be a potential limitation because the addition of terms does not decrease the expected error significantly. However, these results may also be viewed as an advantage since 14 terms in a 20 term KL expansion (for example) yields similar errors as 18 terms in a 20 terms KL expansion. Thus, fewer terms may be used in the decomposition with a neglible loss in accuracy. However,  as the results show, no sacrifice in accuracy is necessary if a higher interpolation level is used. This would simply amount to more precomputation steps in which additional sparse grid points are considered for the Smolyak interpolant.

We also test the sensitivity of the approach with respect to the number of terms $J$ which are kept in Eq.~\eqref{greeninterp}. All previous results in this subsection were obtained from a value of $J=1$ in the expansions; however, it is fitting to offer a comparison between solutions obtained by keeping more terms in the expansion. For the following comparisons we use the same level 1 interpolation results as seen in Fig.~\ref{collstatplot} and test the results corresponding to $J=1$ and $J=4$. In Fig.~\ref{Jcompare} we note that the method does not exhibit sensitivity with respect to $J$. This is again due to the fact that the collocation error is dominant for a level $L=1$ interpolation. We see that the mean errors may be slightly decreased by keeping more terms in the series expansion, yet this increase in accuracy is subtle. This may be attributed to the fact that the $L=1$ Green's function interpolant $I_m G(\Theta_0)$ in  Eq.~\eqref{greeninterp} does not guarantee faster convergence depending on the number of terms kept in the series expansion. More specifically, the dominant collocation error is inhereted through the iterative procedure. As more terms do not yield a significant gain in accuracy, keeping less terms is preferable in this low interpolation level setting. However, as the splitting error becomes dominant (i.e., a higher interpolation level is considered) using more $J$ terms would, in fact,  result in more pronounced accuracy. We next elaborate on some additional effects of a higher interpolation level.

\begin{figure}
\centering
   \includegraphics[width=0.9\textwidth]{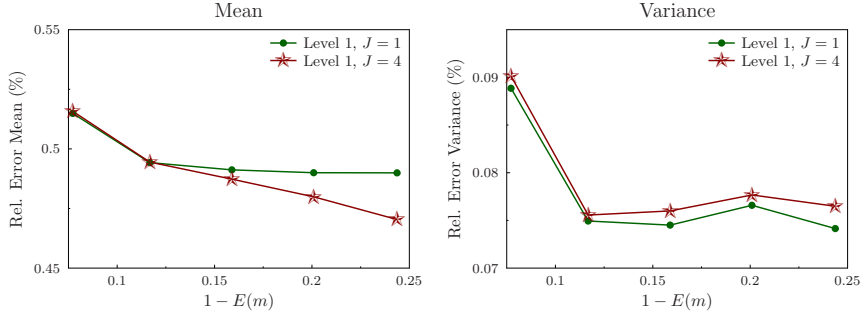}
   \caption{ Comparison between the expected value (left), and variance (right) of the relative error quantities; $L=1$, $J=1, 4$ }
   \label{Jcompare}
\end{figure}

To finish this section we offer detailed relative error comparisons for $||| u_h - \tilde{u}_{J, h} ||| / ||| u_h |||$ (total error), $||| u_h - u_{J, h} ||| / ||| u_h |||$ (splitting error), and $||| \tilde{u}_{J, h} - u_{J, h} ||| / ||| u_h |||$ (collocation error). In doing so, the aim is to solidify the contention that a higher interpolation level indeed diminishes the the collocation error. In turn, if the neglible total error ($< 1 \%$) which results from a lower level interpolation (e.g., $L=1$) is not suitable, a higher level interpolation (e.g., $L=3$) may be used to negate the collocation error. To reiterate, $u_{J, h}$ denotes a solution obtained from the proposed MsFEM approach with Monte Carlo sampling, and $\tilde{u}_{J, h}$ denotes a solution obtained from the proposed MsFEM approach with parameter reduction. See Fig.~\ref{methodcompare} for an illustration of the various errors. We introduce a slight abuse of the original notation from Subsection~\ref{paramreduc}, and use $e$ to denote the total {\it relative} error, $e_{spl}$ to denote the {\it relative} splitting error, and $e_{col}$ to denote the {\it relative} collocation error in Fig.~\ref{methodcompare}. In the left hand side of Fig.~\ref{methodcompare} we note that the smallest errors result from the proposed MsFEM combined with Monte Carlo sampling (i.e., the splitting error is small). In addition, the total error and collocation error are comparable. Thus, we conclude that the collocation approach yields the dominant soure of error for a low level interpolant. However, we note the significant difference in the right hand side of Fig.~\ref{methodcompare}. In this case we note two important factors. First, the total error from a $L=3$ interpolant is smaller than its $L=1$ counterpart (as expected). Furthermore, we see that the splitting error is now the dominant source of error, and the collocation error is much smaller. In other words, an increase in the interpolation level accomplishes the task of significantly reducing the collocation error. Thus, we conclude that for a higher interpolation level the parameter reduction approach behaves much like the standard Monte Carlo counterpart due to the minimal effect of collocation error.

\begin{figure}
\centering
   \includegraphics[width=1.0\textwidth]{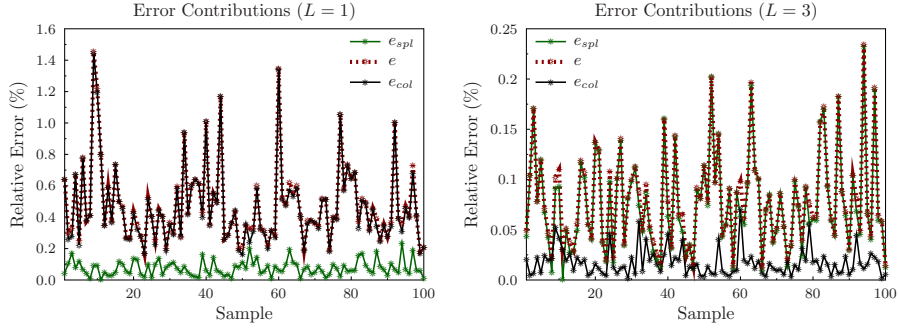}
   \caption{ Comparison between the relative error quantities resulting from respective methods; $m=16$, $J=1$, $L=1$ (left), $L=3$ (right) }
   \label{methodcompare}
\end{figure}

\section{Conclusions}
In this paper we present a new MsFEM
 approach for solving  elliptic equations with coefficients that vary on many length scales and contain
 uncertainties. Through considering a coefficient decomposition combined with a  Green's function approach, we are able to construct new MsFEM
  solutions which closely recover traditional MsFEM solutions. In a deterministic setting, rigorous error estimates and bounds are first presented for a representative basis function within the MsFEM approximation space. Using the initial basis function results, we offer a rigorous error analysis describing the behavior of the elliptic solutions that are sought in the space that is spanned by the new multiscale basis functions. Under appropriate assumptions on the coefficient splitting configuration, we are ultimately able to construct approximate solutions which nicely converge to a benchmark solution. The basis function and elliptic solution analysis are thoroughly verified through a number of representative numerical examples. In a stochastic setting, the proposed solution method is shown to reduce the number of sample solutions that must be constructed for assessing the statistical behavior of the system. In particular, the splitting gives rise to a situation where the random parameter dimension can be reduced, and where stochastic collocation becomes an efficient alternative to direct Monte Carlo sampling. As the parameter reduction sampling approach involves a number of pre-computation steps that are completely independent, this approach is especially desirable in a parallel setting. Analogous error bounds are derived for the stochastic problem, and the successful performance of the proposed method is verified through a variety of numerical examples.

\section*{Acknowledgments}
We would like to formally thank Professor V. Ginting for helpful discussions on the theory and implementation of the proposed method.
His expertise in multiscale methods made the current work more manageable. L. Jiang also thanks  Professor Y. Efendiev and
Dr. J. David Moulton for useful discussions and comments. We thank the reviewers for their comments to improve the paper.

\appendix
\section{proof of Lemma \ref{lem-xi-j}}
\label{app1}
Let $\tx_0$ solve equation (\ref{eq-xi0}). Then integration by parts gives
\[
(k_0\nabla \tx_0, \nabla \tx_0)=(k_1\nabla(\Pi-I)l, \nabla \tx_0)=({k_1\over k_0} k_0\nabla(\Pi-I)l, \nabla \tx_0).
\]
This implies that
\begin{equation}
\label{ineq-xi0}
\|\sqrt{k_0} \nabla \tx_0\|_{0,K} \leq \eta_K \|\sqrt{k_0} \nabla(\Pi-I)l\|_{0,K} \leq 2 \eta_K \|\sqrt{k_0} \nabla l\|_{0,K}.
\end{equation}
Let $\tx_j$ ($j=1,2,\cdots$) solve Eq.~(\ref{eq-xi-j}). Then a similar argument shows
\begin{equation}
\label{ineq-xi-j}
\|\sqrt{k_0} \nabla \tx_j\|_{0,K} \leq \eta_K \|\sqrt{k_0} \nabla \tx_{j-1}\|_{0,K}.
\end{equation}
Using (\ref{ineq-xi-j}) recursively and (\ref{ineq-xi0}), we have
\[
\|\sqrt{k_0} \nabla \tx_j\|_{0,K} \leq \eta_K^j \|\sqrt{k_0} \nabla \tx_0\|_{0,K}\leq 2 \eta_K^{j+1} \|\sqrt{k_0} \nabla l\|_{0,K},    \quad j=0,1,2, \cdots.
\]
The proof is completed.

\section{proof of Lemma \ref{singlethm}}
\label{app2}
Since  $\phi_J=(I-\Pi)l+\xi_J$ and $\phi=(I-\Pi)l+\xi$, it suffices to show that
\[
\lim_{J\rightarrow \infty}|||\xi_J-\xi|||_{K} =0.
\]
By adding Eq.~(\ref{eq-xi0}) and the sequence of equations in $(\ref{eq-xi-j})$, it follows that
\begin{eqnarray}
\label{thm1-eq1}
\begin{cases}
\begin{split}
-\nabla \cdot \left( k_0\nabla(\sum_{j=0}^J\tx_j) \right)&=\nabla \cdot \left( k_1\nabla (\sum_{j=0}^{J-1}\tx_j) \right)-\nabla \cdot (k_1\nabla (\Pi-I)l)  \quad \text{in $K$}\\
\sum_{j=0}^J\tx_j&=0   \quad \text{on $\pK$}.
\end{split}
\end{cases}
\end{eqnarray}
Because $\sum_{j=0}^{J-1}\tx_j=\xi_J-\tx_J$, Eq.~(\ref{thm1-eq1}) is simplified to
\begin{eqnarray}
\label{thm1-eq2}
\begin{cases}
\begin{split}
-\nabla \cdot (k \nabla \xi_J) &=-\nabla \cdot (k_1\nabla \tx_J)-\nabla \cdot (k_1\nabla (\Pi-I)l)  \quad \text{in $K$}\\
\xi_J &=0   \quad \text{on $\pK$}.
\end{split}
\end{cases}
\end{eqnarray}
Applying integration by parts to (\ref{thm1-eq2}),  we have
\begin{eqnarray*}
\begin{split}
(k\nabla \xi_J, \nabla \xi_J)&=(k_1\nabla \tx_J, \nabla \xi_J) +(k_1\nabla (\Pi-I)l, \nabla \xi_J)\\
&= \left( \frac{k_1}{\sqrt{kk_0}} \sqrt{k_0}\nabla \tx_J, \sqrt{k} \nabla \xi_J \right)+ \left( \frac{k_1}{\sqrt{kk_0}} \sqrt{k_0} \nabla (\Pi-I)l, \sqrt{k} \nabla \xi_J \right).
\end{split}
\end{eqnarray*}
As a consequence, it follows immediately that
\begin{eqnarray*}
\begin{split}
|||\xi_J|||_K &\leq \| \frac{k_1}{\sqrt{kk_0}} \|_{L^{\infty}(K)} \|\sqrt{k_0} \nabla \tx_J\|_{0, K} +\|\frac{k_1}{\sqrt{kk_0}}\|_{L^{\infty}(K)} \|\sqrt{k_0}\nabla (\Pi-I)l\|_{0,K}\\
& \leq 2\|\frac{k_1}{\sqrt{kk_0}}\|_{L^{\infty}(K)} (\eta_K^{J+1}+1) \|\sqrt{k_0}\nabla l\|_{0,K}\\
&\longrightarrow 2\|\frac{k_1}{\sqrt{kk_0}}\|_{L^{\infty}(K)} \|\sqrt{k_0}\nabla l\|_{0,K}  \quad \text{as $J \rightarrow \infty$},
\end{split}
\end{eqnarray*}
where we have used Lemma \ref{lem-xi-j} in the last step. Hence,  $\xi_J$ is convergent, and $\phi_J$ is convergent as well.
Subtracting Eq.~(\ref{thm1-eq2}) from Eq.~(\ref{eq-xi}), we have
 \begin{eqnarray}
\label{thm1-eq3}
\begin{cases}
\begin{split}
-\nabla \cdot (k \nabla (\xi-\xi_J)) &=\nabla \cdot (k_1\nabla \tx_J)  \quad \text{in $K$}\\
\xi- \xi_J &=0   \quad \text{on $\pK$}.
\end{split}
\end{cases}
\end{eqnarray}
Performing integration by parts and using the Cauchy-Schwarz inequality for Eq. (\ref{thm1-eq3}), then we obtain
\begin{eqnarray*}
\begin{split}
|||\xi-\xi_J|||_K &\leq \|\frac{k_1}{\sqrt{kk_0}}\|_{L^{\infty}(K)} \|\sqrt{k_0} \nabla \tx_J\|_{0, K}\\
&\leq 2  \|\frac{k_1}{\sqrt{kk_0}}\|_{L^{\infty}(K)} \eta_K^{J+1} \|\sqrt{k_0} \nabla l\|_{0, K}\\
& \longrightarrow 0 \quad \text{as $J\rightarrow 0$}.
\end{split}
\end{eqnarray*}
This completes the proof.

\section{proof of Proposition  \ref{prop3.2}}
\label{app3}
Because  $\phi_J=(I-\Pi)l+\xi_J$ and $\phi=(I-\Pi)l+\xi$, it suffices to show
\[
|||\xi-\xi_J|||_K \leq C_l\frac{2 b_1}{\sqrt{a_0}} \eta_K^{J+2}.
\]
In fact,  the proof of Lemma \ref{singlethm} implies that
\begin{eqnarray*}
\begin{split}
&|||\xi-\xi_J|||_K \leq  2  \|\frac{k_1}{\sqrt{kk_0}}\|_{L^{\infty}(K)} \eta_K^{J+1} \|\sqrt{k_0} \nabla l\|_{0, K}\\
&\leq 2\|{k_1\over k_0}\sqrt{{k_0\over k}}\|_{L^{\infty}(K)}\eta_K^{J+1} \|\sqrt{k_0}\|_{L^{\infty}(K)}\|\nabla l\|_{0,K}
= C_l\frac{2 b_1}{\sqrt{a_0}} \eta_K^{J+2}.
\end{split}
\end{eqnarray*}
This completes the proof.


\begin{thebibliography}{10}
\bibitem{akl06}
{\sc J.~Aarnes, S. ~Krogstad, K. A. ~Lie},  {\em A hierarchical multiscale method for two-phase flow based
on upon mixed finite elements and nonuniform coarse grids}, Multiscale Modeling and Simulation,
5  (2006), pp.~337-363.

\bibitem{ab05}
 {\sc G. Allaire and R. Brizzi},
{\em A multiscale finite element method for numerical homogenization},
Multiscale Modeling and Simulation, 4 (2005),  pp.~790--812.



\bibitem{mortar}
{\sc T.~Arbogast, G.~Pencheva, M.~F.~Wheeler, and I.~Yotov}, {\em A multiscale mortar mixed finite element method},  Multiscale Modeling and Simulation 6 (2007), pp.~319--346.

\bibitem{bco94}
{\sc I.~Babu$\breve{s}$ka, G. Caloz and E. Osborn},
{\em Special finite element methods for a class of second order
elliptic problems with rough coefficients},
SIAM J. Numer. Anal. 31 (1994), pp.~945--981.

\bibitem{bl11}
{\sc I.~Babu$\breve{s}$ka  and R. Lipton}, {\em Optimal local approximation spaces for generalized finite element methods with application to multiscale problems},
Multiscale Modeling and Simulation, 9 (2011), pp. ~373--406.


\bibitem{bnt07}
 {\sc I.~Babu$\breve{s}$ka, F. ~Nobile and  G. ~Zouraris}, {\em A  stochastic collocation
method for  elliptic partial differential equations with random input data}, SIAM J.
  Numer. Anal., 45  (2007), pp.~1005--1034.


\bibitem{bnr00}
{\sc  V.~Barthelmann, E. ~Novak and K. ~Ritter},
{\em High dimensional polynomial interpolation on sparse grids},
 Advanced in Compuational Mathematics 12 (2000), pp.~273--288.




\bibitem{bs94}
{\sc S. ~Brenner and L. ~Scott}, {\em The mathematical theory of finite element methods}, Springer, Berlin,  1994.



\bibitem{deh08}
{\sc P. Dostert, Y. Efendiev,  and T.Y. Hou},
{\em Multiscale finite element methods for stochastic porous media flow equations and application to uncertainty quantification},
Comput. Methods Appl. Mech. Engrg. 197 (2008), pp.~3445--3455.

\bibitem{ee03}
{\sc  W.~E and B.~Engquist}, {\em The heterogeneous multi-scale methods}, Comm.
  Math. Sci., 1 (2003),  pp.~87--133.


\bibitem{victor}
{\sc Y.~Efendiev, V.~Ginting, T.~Hou, and R.~Ewing}, {\em Accurate
multiscale finite element methods for two-phase flow simulations},
J. Comput.Phys. 220 (2006),  pp.~155--174.



\bibitem{Schwab}
{\sc P.~Frauenfelder, C.~Schwab, and R.~A.~Todor}, {\em Finite elements
for elliptic problems with stochastic coefficients},
Comput. Methods Appl. Mech. Engrg. 194 (2005),  pp.~205--228.


\bibitem{gz09}
{\sc B.~Ganapathysubramanian and N.~Zabaras}, {\em A stochastic multiscale framework for modeling flow
through random heterogeneous prous media}, J. Comput. Physics. 228 (2009),  pp.~591-618.


\bibitem{gmp10}
{\sc V.~Ginting, A. ~Malqvist and M. ~Presho}, {\em A novel method for solving multiscale elliptic problems with randomly
perturbed data}, Multiscale Modeling and Simulation, 8 (2010), pp.~977--996.


\bibitem{Hou}
{\sc T. Y. Hou and X.-H. Wu}, {\em A multiscale finite element
method for elliptic problems in composite materials and porous
media}, J. Comput. Phys. 134 (1997), pp.~169-189.

\bibitem{HouII}
{\sc T. Y. Hou and X.-H. Wu, and Z. Cai}, {\em Convergence of a
multiscale finite element method for elliptic problems with rapidly
oscillating coefficients}, Math. Comp., 68 (1999),  pp.~913-943.

\bibitem{HughesII}
{\sc T. Hughes, G. R. Feij\'{o}o, L. Mazzei and J.-B.
Quincy}, {\em The variational multiscale method - a paradigm for
computational mechanics}, Comput. Methods Appl. Mech. Engrg. 166
(1998), pp.~3-24.

\bibitem{jennylt03}
{\sc P. Jenny, S. H. Lee, and H. Tchelepi}, {\em Multi-scale finite volume
  method for elliptic problems in subsurface flow simulation}, J. Comput.
  Phys., 187 (2003), pp.~47--67.


\bibitem{jem10}
{\sc L. Jiang, Y. Efendiev and I. ~Mishev}, {\em Mixed multiscale finite
element methods using approximate global information based on partial upscaling},
Comput. Geosci., 14 (2010), pp.~319--341.


\bibitem{jeg07}
{\sc L.~Jiang, Y.~Efendiev, and V.~Ginting},
{\em Multiscale methods for parabolic equations with continuum spatial scales,}
DCDS Series B, 8 (2007), pp.~833--859.


\bibitem{jml10}
{\sc  L. Jiang, I. Mishev and Y. Li}, {\em Stochastic mixed multiscale finite element methods and their applications in random porous media},  Comput. Methods Appl. Mech. Engrg., 199 (2010),
 pp.~2721--2740.




\bibitem{Kleiber}
{\sc M.~Kleiber and T.~D.~Hien}, {\em The Stochastic Finite Element
Method},  Wiley, New York, 1992.






\bibitem{ntw08}
{\sc F.~Nobile, R. ~Tempone and C. G. ~Webster}, {\em A sparse grid
stochastic collocation method for partial differential  equations with random
input data}, SIAM J. Numer. Anal., 46 (2008), pp.~2309--2345.


\bibitem{sm63}
{\sc  S.~Smolyak},
{\em Quadrature and interpolation formulas for tensor products of certain classes of functions}, Soviet Math. Dokl. 4, 1963, pp.~240--243.

\bibitem{xiu10}
{\sc D. Xiu}, {\em Numerical Methods for Stochastic Computations: A Spectral Method Approach},  Princeton University Press, 2010.




\bibitem{wheeler}
{\sc M.~F.~Wheeler, T.~Wildey, and I.~Yotov}, {\em A multiscale preconditioner for stochastic mortar mixed finite elements},  Comput. Methods Appl. Mech.
Engrg., 200 (2011), pp.~1251--1262.

\end{thebibliography}
\end{document}